\documentclass[10pt,dvipsnames]{article}

\usepackage{amsmath} 
\usepackage{amssymb}
\usepackage{amsthm}
\usepackage{graphicx}
\usepackage{hyperref}
\hypersetup{citecolor = blue,colorlinks,
			linkcolor = black,
			urlcolor = blue}
\usepackage{enumitem}
\usepackage[capitalise]{cleveref}
\usepackage{mathrsfs}
\usepackage{tikz-cd}
\usepackage{xcolor}
\usepackage[toc,page]{appendix}

\usepackage[margin=1.6 in, top=1.2 in, bottom=1.1 in]{geometry}

\newtheorem{theorem}{Theorem}[section]
\newtheorem*{theorem*}{Theorem}
\newtheorem{lemma}[theorem]{Lemma}
\newtheorem*{lemma*}{Lemma}
\newtheorem{proposition}[theorem]{Proposition}
\newtheorem*{proposition*}{Proposition}
\newtheorem{corollary}[theorem]{Corollary}
\newtheorem*{corollary*}{Corollary}
\newtheorem{conjecture}[theorem]{Conjecture}
\newtheorem*{conjecture*}{Conjecture}
\newtheorem{definition}[theorem]{Definition}
\newtheorem{question}[theorem]{Question}

\newtheorem*{example*}{Example}
\theoremstyle{definition}
\newtheorem{remark}[theorem]{Remark}
\newtheorem{example}[theorem]{Example}

\theoremstyle{plain}
\newtheorem*{namedthm}{\namedthmname}
\newcounter{namedthm}
\makeatletter
	\newenvironment{named}[2]
	{\def\namedthmname{#1}
	\refstepcounter{namedthm}
	\namedthm[#2]\def\@currentlabel{#1}}
	{\endnamedthm}
\makeatother

\newcommand{\define}[1]{\textit{#1}}
\renewcommand{\P}{\mathbb{P}}
\newcommand{\R}{\mathbb{R}}
\newcommand{\Z}{\mathbb{Z}}
\newcommand{\Q}{\mathbb{Q}}
\newcommand{\N}{\mathbb{N}}

\newcommand{\T}{\mathbb{T}}
\newcommand{\E}{\mathbb{E}}

\newcommand{\CD}{\mathcal{D}}

\newcommand{\dens}{\mathsf{d}}
\newcommand{\udens}{\overline{\mathsf{d}}}
\newcommand{\ldens}{\underline{\mathsf{d}}}
\newcommand{\ubdens}{\mathsf{d}^*}

\renewcommand{\emptyset}{\varnothing}

\newcommand{\Cesaro}{Ces\`{a}ro}
\newcommand{\Erdos}{Erd\H{o}s}

\newcommand{\Folner}{F\o{}lner}

\newcommand{\Holder}{H\"{o}lder}

\newcommand{\Poincare}{Poincar\'{e}}
\newcommand{\Sarkozy}{S\'{a}rk\"{o}zy}
\newcommand{\Szemeredi}{Szemer\'{e}di}
\newcommand{\Turan}{Tur{\'a}n}

\newcommand{\one}{\boldsymbol{1}}

\newcommand{\FS}{\mathsf{FS}}

\newcommand{\lp}{\mathsf{L}}

\DeclareMathOperator{\fpl}{\mathsf{FPL}}
\DeclareMathOperator{\fpr}{\mathsf{FPR}}

\newcommand{\ultra}[1]{\mathsf{#1}}

\newcommand{\ghk}{|\!|\!|}
\newcommand{\ghkn}[1]{\ghk {#1} \ghk}
\newcommand{\un}[2]{{\mathsf{U}^{#1}(#2)}}

\begin{document}

\author{Bryna Kra\and Joel Moreira\and Florian K.~Richter\and Donald Robertson}

\title{\textbf{Problems on infinite sumset configurations in the integers and beyond}}
\maketitle

\begin{abstract}
In contrast to finite arithmetic configurations, relatively little is known about which infinite patterns can be found in every set of natural numbers with positive density. 
Building on recent advances showing infinite sumsets can be found, we explore numerous open problems and obstructions to finding other infinite configurations 
in every set of natural numbers with positive density. 
\end{abstract}
{
\renewcommand{\baselinestretch}{0}
\tableofcontents
\renewcommand{\baselinestretch}{1.0}
}

\section{Introduction}

A central topic in arithmetic combinatorics is the study of additive structures in the integers, such as arithmetic progressions and sumsets. The foundations for the area were laid by early results in Ramsey Theory.
Van der Waerden~\cite{VDW} showed that in any finite coloring of the integers, there are arbitrarily long monochromatic arithmetic progressions. 
Thereafter \Erdos{} and \Turan{}~\cite{ET36} conjectured and
\Szemeredi{}~\cite{S} proved a density analog of van der Waerden's theorem:  every set of natural numbers with positive density contains arbitrarily long arithmetic progressions.
Many other finite arithmetic structures have been studied from a combinatorial perspective, including solutions to systems of linear Diophantine equations~\cite{Rado33, Schur16}, polynomial progressions~\cite{bergelson-leibman}, and sum-product configurations~\cite{Moreira}, to name a few.
The techniques developed to tackle these problems are robust enough to play a role in many other contexts, leading to advances on multi-dimensional patterns~\cite{Furstenberg-Katznelson}, progressions in the primes~\cite{Green_Tao08, Green_Tao10},  and patterns in amenable groups~\cite{austin2016}.
There have been many and varied developments besides those mentioned here; for more we recommend as a starting point the surveys and monographs~\cite{Bergelson96,conlon-fox-zhao,f-problems,Furstenberg-book,Graham_Rothschild_Spencer80,HK-book}.

This survey is about the arithmetic combinatorics of \textit{infinite} configurations, a relatively young branch of arithmetic combinatorics that complements and contrasts both classical and ongoing work on finite configurations. 
What infinite configurations should one expect to be present in every set of natural numbers with positive upper density? 
It follows from Ramsey's theorem~\cite{ramsey} that, no matter how the elements of $\N$ are colored using finitely many colors, there is an infinite set $B \subset \N$ such that
\[
\{ b_1 + b_2 : b_1,b_2 \in  B,~ b_1 \ne b_2 \}
\]
is monochromatic.
\Erdos{} made several conjectures that serve as density analogs of this result.
For example, amongst them was the statement that every positive density set of natural numbers contains the sum
\[
B+C := \{ b+c:b\in B,~ c\in C \}
\]
of two infinite sets $B,C \subset \N$.
Motivated by recent developments~\cite{ bergelson,choi-heath,DGJLLM15, Host,KMRR,KMRR2,MRR,Nathanson} on \Erdos{}'s conjectures, we explore various new directions of research pertaining to infinite sumset phenomena in the integers and beyond.  
We supply examples that provide natural obstructions to infinite sumset configurations and collect open problems of interest. 
Our focus is predominantly on questions and conjectures within the area of infinitary density combinatorics; for a survey of problems in infinitary partition regularity see~\cite{HLS03}.

\subsubsection*{Notions of density}
\label{subsec:density}

Throughout we let $\N=\{1,2,3,\ldots\}$ denote the positive integers and for $k\in\N$, let $k\N=\{kn:n\in\N\}$ denote the set of all positive multiples of $k$.
There are various notions of density we can consider on $\N$, including upper and lower natural density, and upper Banach density. We recall here those that we use most frequently.

If $A$ is a set of natural numbers, define its \define{upper Banach density} $\ubdens(A)$ to be
\begin{equation}
\label{eqn:positive_upper_density}
\ubdens(A):=\limsup_{N -M\to \infty} \frac{\big|A\cap \{M+1,M+2,\dots,N\}\big|}{N-M}.
\end{equation}
The \define{upper density} $\udens(A)$ and \define{lower density} $\ldens(A)$ of $A\subset\N$ are defined respectively as
\[
\udens(A)=\limsup_{N\to\infty}\frac{|A\cap\{1,\ldots,N\}|}{N}
\quad\text{and}\quad 
\ldens(A)=\liminf_{N\to\infty}\frac{|A\cap\{1,\ldots,N\}|}{N}.
\]
If $\udens(A) = \ldens(A)$, then we call this number the \define{natural  density} of $A$ and denote it by $\dens(A)$.
We formulate most of our conjectures and statements in terms of upper Banach density, the weakest measurement of largeness among these notions, but for most questions and conjectures one can ask about the analog under a stronger density assumption. 

In some settings, we  use  density with respect to a \define{\Folner{} sequence}  in $\N$, meaning a sequence $\Phi=(\Phi_N)_{N\in\N}$ of finite subsets of  $\N$ satisfying 
$$
\lim_{N\to\infty} \frac{|(\Phi_N+t)\cap \Phi_N|}{|
\Phi_N|} = 1
$$
for all $t\in\N$, where $\Phi_N+t$ denotes the shift of the set $\Phi_N$ by $t$. 

Analogous to upper and lower density, the \define{upper} and \define{lower density along $\Phi$} of a set $A\subset\N$ are defined as
\[
\udens_\Phi(A) = \limsup_{N\to\infty}\frac{|A\cap \Phi_N|}{|\Phi_N|}
\quad
\text{and}
\quad 
\ldens_\Phi(A)=\liminf_{N\to\infty}\frac{|A\cap \Phi_N|}{|\Phi_N|}.
\]
When the quantities above are equal, we denote the common value by $\dens_\Phi(A)$ and call it the \define{density} of $A$ along $\Phi$.

As needed, we make use of the natural  analogs of \Folner{} sequences and densities along \Folner{} sequences in the setting of $\N^d$ for $d\geq 2$.
See \cref{sec:amenable} for the precise definition of \Folner{} sequences in even more general settings.

\subsubsection*{Acknowledgements}
We are grateful to the organizers of \textit{Nilpotent structures in topological dynamics, ergodic theory and combinatorics} at IMPAN Mathematical Research and Conference Center, of the \textit{Zurich Dynamics Conference} at the University of Zurich, and of \textit{Pointwise Ergodic Theory \& Connections} at King's College London,  from whose hospitality and support we have benefited while preparing this manuscript. 
Part of this work was done while JM and FKR were members of IAS in Princeton and we thank the institute for the hospitality and acknowledge National Science Foundation grant  DMS-1926686.
Part of this work was done while JM, FKR and DR participated in an ICMS Research in Groups program, and we thank them for their support.
BK acknowledges support from the National
Science Foundation grants DMS-2054643 and DMS-2348315. We thank Thomas Bloom, Felipe Hern{\'a}ndez Castro, and Andrew Granville for fruitful discussions and comments on earlier drafts of this article, and thank the referees for their many comments improving the article.

\section{Infinite sumsets in sets of integers with positive density}
\label{sec:sumsets_integers}

\subsection{Searching for a density version of Hindman's theorem}
\label{sec:density-Hindman}

We begin by exploring what types of infinite sumset configurations can be found in subsets of the integers with positive upper Banach density.  
Results of this type can be viewed as density analogs of one of the central results in additive combinatorics: Hindman's theorem.
Given a set $B\subset\N$, its \define{finite sums set} is defined by
\begin{equation}
    \label{def:FS}
\FS(B) = \left\{ \sum_{n \in F}  n :   F \subset B\textup{ with } 0< |F| < \infty \right\}
\end{equation}
and consists of all natural numbers obtained by adding together finitely many distinct elements in $B$.
When $B$ is infinite we say that $\FS(B)$ is an \emph{IP-set}\footnote{Depending on the source, the acronym IP stands for \emph{infinite parallelepiped}~\cite{FW79} or \emph{idempotent}~\cite{FK85}, where the latter relates to the Stone-{\v C}ech compactification of the integers.}.

\begin{named}{Hindman's Theorem}
{\cite{Hindman-1974}}
If the natural numbers are partitioned into finitely many cells then at least one of these cells contains an IP-set.
\end{named}

It follows from Ramsey's theorem~\cite{ramsey} that for any finite coloring of $\N$ and any $k \in \N$ there exists an infinite set $B\subset \N$ such that 
\[
\left\{ \sum_{n \in F} n : F \subset B \textup{ with } |F| = k \right\}
\]
is monochromatic. However, before Hindman's theorem it was not even known that infinite monochromatic configurations of the form $B\cup \{b_1+b_2: b_1,b_2\in B,\,b_1\neq b_2\}$ exist for any finite coloring of $\N$ (see~\cite[Question~12]{HLS03} and the surrounding discussion).

In a first attempt to formulate a density version of Hindman's theorem, \Erdos{} asked whether every set $A\subset\N$ with positive upper density has a shift\[
A-t := \{ n \in \N : t + n \in A \}
\]
for some integer $t\geq 0$ that contains an IP-set.
Note that the shift is necessary because the set of odd numbers has density $1/2$, yet it does not contain a triple of the form $\{x,y,x+y\}$.
However, Straus provided an example (not published by Straus, but published by others, for example in~\cite[Theorem 11.6]{Hindman79}), showing that \Erdos{}'s question has a negative answer.

\begin{example}[Straus example]
\label{example_straus}
Let $(p_t)_{t\geq 1}$ be a sufficiently fast growing sequence of primes and let
$$A:=\N\setminus\left(\bigcup_{t\in\N}(p_t\N+t)\right)$$
Then $\dens(A)\geq1-\sum_{t=1}^\infty\frac1{p_t}>0$ but any shift $A-t$ of $A$ contains only finitely many multiples of $p_t$ and hence cannot contain an IP-set.
\end{example}

Undeterred by this example and in light of the ongoing advances surrounding \Szemeredi{}'s and Hindman's theorems at the time, \Erdos{}~\cite[Page~305]{erdos-1975} wrote: 
\begin{quote}
I have tried to formulate a conjecture which would be in the same relation to Hindman's theorem as \Szemeredi{}'s theorem is to van der Waerden's. I have not been very successful so far.
Perhaps the following result holds.
Let $a_1 < a_2 < \dots$ be a sequence of integers with positive upper density.
Then there is an integer $t$ and an infinite subsequence $a_{i_1} < a_{i_2} < \dots$ so that all the sums $a_{i_r} + a_{i_s} + t$ are again $a$'s.
\end{quote}
\Erdos{}'s proposed result, together with the weaker conjecture from~\cite[Pages 57--58]{erdos-1977} and~\cite[Page 105]{erdos-1980}, were proved by the authors.

\begin{theorem}[{\cite[Theorem 1.2, part (i)]{KMRR2}}]
\label{thm_BBt}
    Every set $A\subset\N$ of positive upper Banach density contains a sumset
\[t+\big\{ b_1+b_2:b_1,b_2\in B,~b_1\neq b_2 \big\}\]
for some shift $t\geq0$ and some infinite set $B \subset A$.
\end{theorem}

We note that there are examples showing that one cannot remove the restriction $b_1\neq b_2$.
One such example was communicated to us by Steven Leth.

\begin{example}
\label{eg:not2b}
For $c\in(1,2)$ consider the set
\begin{equation}
\label{eq_archemedianobstructions}
A
=
\N \cap \bigcup_{n\in\N}\left[4^{n},c\cdot4^{n}\right),
\end{equation}
which has $\ubdens(A)=1$ and $\ldens(A)>0$.
The set $A$ does not contain any set of the form $\{b, 2b\}$ for $b \in \N$.
When $b$ is much larger than $b'$ one cannot find $\{b+b',2b\}$ in $A$ either.
One can use this observation to show that $A$ does not contain an unrestricted sumset $B+B$ for an infinite set $B$.
In fact, every shift of $A$ has the same property, and so $A$ cannot contain any set of the form $B+B+t$ with $B \subset \N$ infinite and $t\geq 0$ integer.
The proofs of these statements are deferred to the more general setting of \cref{example_nextexamplebutnowitsgone}.
\end{example}

In an attempt to meet \Erdos{}'s aspirations for a density version of Hindman's theorem, we propose several conjectures that extend \cref{thm_BBt} in different directions without encroaching on the example of Straus. 
One of the most natural conjectures along these lines is the following.

\begin{conjecture}[{\cite[Conjecture~1.5]{KMRR2}}]
\label{q:bbtbbbtbbbbt}
Fix $k \in \N$.
Every set of positive upper Banach density $A\subset \N$ has a shift $A-t$ which contains the set 
\[\left\{ \sum_{n \in F}  n :   F \subset B\textup{ with } 0< |F| \leq k \right\}\]
for some integer $t\ge 0$ and some infinite set $B \subset \N$.
\end{conjecture}

We stress that one cannot refine \cref{q:bbtbbbtbbbbt} by choosing $t$ independent of $k$ because the example of Straus does not contain such configurations.

The case $k=2$ of \cref{q:bbtbbbtbbbbt}, which corresponds to a slight variation of \cref{thm_BBt}, is proved in~\cite[Theorem 1.2, part (ii)]{KMRR2}.
For $k\geq 3$ even the following special case of \cref{q:bbtbbbtbbbbt} remains open. 
\begin{conjecture}
\label{conj:bbbt}
Fix $k \in \N$.
Every set of positive upper Banach density $A\subset N$ has a shift $A-t$ which contains a sumset
\begin{equation}
\label{eqn:BBBBBB}
B^{\oplus k}:=\left\{\sum_{n\in F}n:F\subset B\text{ with }|F|=k\right\}
\end{equation}
for some integer $t \ge 0$ and some infinite set $B \subset \N$.
\end{conjecture}

Throughout the rest of the paper, we  continue using the notation $B^{\oplus k}$ introduced in \cref{conj:bbbt}, and for the case $k=2$, where there is no ambiguity, we use the more explicit notation $B \oplus B$ instead of  $B^{\oplus 2}$.

The odd integers make it clear that some shift $t$ is needed in \cref{conj:bbbt}. 
However, it is possible to restrict the value of the shift to the set $\{0,\dots, k-1\}$ by analysing which congruence class modulo $k$ is preferred by the set $A$. In particular, if the intersection $A\cap k\Z$ has positive density, then we can take $t=0$.  We summarize this in the next proposition. 

\begin{proposition}[See also {\cite[Theorem 11.8]{Hindman79}}]
\label{prop:equivalences-for-shifts}
The following are equivalent.
\begin{enumerate}
\item
\cref{conj:bbbt}.
\item
For every $k\in\N$ and every $A\subset\N$ with $\ubdens(A)>0$ there exist an infinite set $B\subset\N$ and integer $t\in\{0,\dots,k-1\}$ with $B^{\oplus k}+t\subset A$.
\item
For every $k\in\N$ and every $A\subset \N$ with $\ubdens(A\cap k\N)>0$ there exists an infinite set $B\subset \N$ with $B^{\oplus k}\subset A$.
\end{enumerate}
\end{proposition}
\begin{proof}
Clearly the second statement implies the first.

To see how the third statement implies the second, let $A\subset\N$ with $\ubdens(A)>0$ and find $t\in\{0,\dots,k-1\}$ such that $A':=A\cap (k\N+t)$ has $\ubdens(A')>0$.
Then $A'-t\subset k\N$ and by the third statement there exists an infinite set $B\subset\N$ with $B^{\oplus k}\subset A'-t$.

To show that the first statement implies the third, let $A\subset k\N$ with $\ubdens(A)>0$ and, using the first statement, find $t\geq0$ and an infinite set $B\subset\N$ such that $B^{\oplus k}+t\subset A\subset k\N$.
Now $B$ must contain $k$ elements which are all congruent modulo $k$, and therefore their sum is a multiple of $k$.
It follows that $t=k\tilde{t}$ for some $\tilde{t}\in\Z$.
Letting $\tilde{B}=B+\tilde{t}$ we conclude that $\tilde{B}^{\oplus k}\subset A$.
\end{proof}
We continue the discussion on the role played by the shift  $t$ in \cref{q:bbtbbbtbbbbt} and related contexts in \cref{sec_uniformity_norms}.

\subsection{Sums of distinct infinite sets}

One can weaken \cref{conj:bbbt} by replacing~\eqref{eqn:BBBBBB} with the sum
\[
B_1+\cdots+B_k := \bigl\{ b_1+\dots + b_k:b_1\in B_1,\dots, b_k\in B_k \}
\]
of $k$ infinite sets $B_1,\dots,B_k \subset \N$. 
It is then unnecessary to require a shift $t$ as it can be absorbed into one of the infinite sets. The $k=2$ case of this weakening was established in 2019.

\begin{theorem}[\cite{MRR}]
\label{thm:MRR}
If $A\subset\N$ has positive upper Banach density, there exist infinite sets $B, C\subset\N$ such that $B+C \subset A$.
\end{theorem}
In~\cite{Host}, Host gave an ergodic theoretic proof of \cref{thm:MRR}, finding the proper ergodic analog for some of the ideas in~\cite{MRR}.
The techniques used in both~\cite{MRR} and~\cite{Host} seem to be unable to handle higher order sumsets.
In~\cite{KMRR} a different ergodic proof of \cref{thm:MRR} was discovered, building on the ideas of~\cite{Host}, which works for higher-order sumsets.

\begin{theorem}[{\cite[Theorem 1.1]{KMRR}}]
\label{thm_BCD}
For every set  $A\subset\N$ of positive upper Banach density and integer $k\in\N$, there exist infinite sets $B_1, \dots, B_k\subset\N$ such that $B_1+\cdots+B_k \subset A$.
\end{theorem}
The conclusion of \cref{thm_BCD} follows from \cref{conj:bbbt} by splitting the set $B$ into $k$ disjoint infinite subsets.
This motivates several variations of the conjectures and questions posed in Section~\ref{sec:density-Hindman}, 
starting with the following.

\begin{conjecture}
\label{conj_B1uptoBk}
Given $A\subset\N$ with positive upper Banach density and $k\in\N$ there are infinite sets $B_1,\dots,B_k\subset\N$ such that $B_1+\cdots+B_i\subset A$ for all  $i=1,\dots,k$.
\end{conjecture}

We note that \cref{conj_B1uptoBk} follows from \cref{q:bbtbbbtbbbbt}: if $B$ and $t$ satisfy the conclusion of the latter, then partitioning $B$ into $k$ disjoint infinite subsets $B=B_1'\cup\cdots\cup B_k'$ and letting $B_1=B_1'+t$ and $B_i=B_i'$ for all $i>1$ we obtain the conclusion of \cref{conj_B1uptoBk}.

One may be tempted to ask whether $A$ contains all sums
\[
B_{i(1)} + \cdots + B_{i(r)}
\]
for all $1 \le i(1) < \cdots < i(r) \le k$ 
for some $r$ with $2\leq r \leq k$.
However, there are local obstructions to this and in particular the set $A=2\N-1$ does not contain such a configuration.

Another possibility is to strengthen \cref{conj_B1uptoBk} by essentially making $k$ infinite.
\begin{conjecture}
    \label{question_infiniteBi}
Given $A\subset\N$ with positive upper Banach density, there is an infinite sequence of infinite sets $B_1,B_2, \ldots \subset\N$ such that $B_1+\cdots+B_i\subset A$ for every $i\in\N$.
\end{conjecture}

We end this section with two related but more ambitious questions. 
While the Straus Example precludes the possibility that a single shift of an arbitrary set $A$ of positive Banach density contains an IP-set, it may be possible to have a single infinite set $B$ and different shifts depending on the number of summands.

\begin{question}\label{question_motherHindman}
Given $A\subset\N$ with positive upper Banach density, does there exist an infinite set $B \subset\N$ such that for every $k\in\N$ there exists a shift $t_k\geq0$ with
\[\left\{ \sum_{n \in F}  n :   F \subset B\textup{ with } 0< |F| \leq k \right\}\subset A-t_k. \]
\end{question}
It is not clear if a positive answer is to be expected, but we were unable to find a counterexample.
Perhaps a slightly weaker version where one is allowed to disregard finitely many elements of $B$ as the number of summands increases has higher chances of a positive answer.
\begin{question}\label{question_sameB}
Given $A\subset\N$ with positive upper Banach density, does there exist an infinite set $B \subset\N$ such that for all $k\in\N$ there is a co-finite set $B'\subset B$ and a shift $t_k\geq0$ with
\[\left\{ \sum_{n \in F}  n :   F \subset B'\textup{ with } 0< |F| \leq k \right\}\subset A-t_k. 
\]
\end{question}

We note that a positive answer to \cref{question_motherHindman} implies all the conjectures listed in this section. See~\cref{fig:diagram} for a schematic.

\begin{figure}[h!]
\centering
\begin{tikzcd}
&  {\cref{question_motherHindman}} \arrow[dl, Rightarrow] \arrow[d, Rightarrow] \\
{\cref{question_infiniteBi}}\arrow[dd, Rightarrow] &  {\cref{question_sameB}}\arrow[d, Rightarrow] \arrow[l, Rightarrow, crossing over]\\
& {\cref{q:bbtbbbtbbbbt}} \arrow[dl, Rightarrow] \arrow[d, Rightarrow] \\
{\cref{conj_B1uptoBk}}\arrow[d, Rightarrow] & {\cref{conj:bbbt}}\arrow[d, Rightarrow]\arrow[dl, Rightarrow]  \\
{\cref{thm_BCD}}\arrow[d, Rightarrow] & {\cref{thm_BBt}}\arrow[dl, Rightarrow]  \\
{\cref{thm:MRR}}&
\end{tikzcd}
\caption{Summary of the relations among the questions (assuming positive answers), conjectures,  and theorems in \cref{sec:sumsets_integers}.}
\label{fig:diagram}
\end{figure}
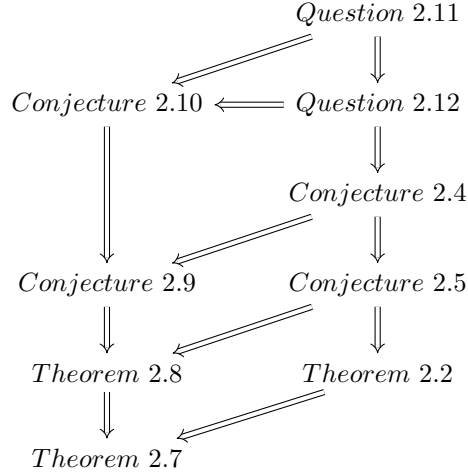

\section{Refinements of sumsets in positive density}
\label{sec:refinements}

\subsection{Constraining the summands}

\cref{thm:MRR} states that any set $A\subset\N$ with $\ubdens(A)>0$ contains a sumset $B+C$ for two infinite sets $B,C\subset\N$.
It is natural to ask whether one can impose certain restrictions on the sets $B$ and $C$.
We first show as a straightforward consequence of \cref{thm:MRR} that, up to a translate, one can draw $B$ and $C$ from any subsets of the integers with bounded gaps. 
Subsets of $\N$ with bounded gaps are called \define{syndetic}. 
Equivalently, a set $S \subset \N$ is syndetic if there exists $h\in\N$ such that $(S-1) \cup (S-2) \cup \cdots \cup (S-h) = \N$.

\begin{proposition}
\label{prop_syndeticB+C}
Fix $A \subset \N$ with $\ubdens(A) > 0$.
For any syndetic subsets $P$ and $Q$ of $\N$ there is $P' \subset P$ infinite and $Q' \subset Q$ infinite and $t \ge 0$ with $P' + Q' + t \subset A$.
\end{proposition}
\begin{proof}
By \cref{thm:MRR}, there exist infinite sets $B,C\subset\N$ such that $B+C\subset A$.
Since $P$ and $Q$ are syndetic, we can find shifts $p,q\in\N$ such that $(P+p)\cap B$ and $(Q+q)\cap C$ are both infinite.
Take
\begin{align*}
P' &= ((P+p)\cap B)-p=P\cap(B-p) \\
Q' &= ((Q+q)\cap C)-q=Q\cap(C-q)
\end{align*} 
and $t=p+q$.
We have that  $P'+Q'\subset B+C-t\subset A-t$.
\end{proof}

The shift $t$  in \cref{prop_syndeticB+C} is needed, as taking both $P$ and $Q$ to be the even numbers and $A$ to be the odd numbers shows.
Next we consider the situation where only $C$ is constrained; in this setup a shift is no longer needed as it can be absorbed into $B$.
As the following proposition shows, in order to be a repository for $C$ a set need not be syndetic across all of $\N$ but it must be ``visible locally along a \Folner{} sequence that sees $A$''.

\begin{proposition}
\label{prop_B+Cindensity}
Let $W\subset\N$ and let $\Phi$ be a \Folner{} sequence for which $\dens_\Phi(W)$ exists. 
The following are equivalent.
\begin{enumerate}
\item
\label{item1ofthepropositiononequivalentcontainment}
For every $A\subset\N$ with $\udens_\Phi(A)>0$,  there exist infinite sets $B\subset\N$, $C\subset W$ such that $B+C\subset A$.
\item
\label{item2ofthepropositiononequivalentcontainment}
$\dens_\Phi(W)>0$.
\end{enumerate}
\end{proposition}
\begin{proof}
The implication
$
\eqref{item2ofthepropositiononequivalentcontainment}
\Rightarrow
\eqref{item1ofthepropositiononequivalentcontainment}
$
follows from a careful analysis of the proof in~\cite{MRR} (or the proofs in~\cite{Host} and in~\cite{KMRR}).

The proof of
$
\eqref{item1ofthepropositiononequivalentcontainment}
\Rightarrow
\eqref{item2ofthepropositiononequivalentcontainment}
$
is by construction.
If $\dens_\Phi(W)=0$ we can find a slowly growing sequence $f\colon \N\to\N$ such that  $\lim_{n\to\infty}  f(n)=\infty$ and such that the set
\[
\tilde W := \bigcup_{w\in W}[w-f(w),w+f(w)]
\]
still has $\dens_\Phi(\tilde W)=0$.
Letting $A=\N\setminus\tilde W$ we see that $\dens_\Phi(A)=1$.
Whenever $A$ contains a sumset $B+C$ and $C\subset W$, if $b\in B$ and $w\in W$ are such that $f(w)>b$, then we cannot have $w\in C$ (as $w+b\notin A$).
It follows that $C$ must be finite, finishing the proof.
\end{proof}

Among other things, \cref{prop_B+Cindensity} implies that not every set $A\subset\N$ with $\dens(A)>0$ contains a sumset $B+C\subset A$ with $B,C\subset\N$ infinite and $C$ consisting entirely of squares or of primes. 
(However, see \cref{sec:recurrence-and-sumset} for positive results pertaining to sumsets with restrictions to squares, primes, or other sparse subsets of $\N$.)

While \cref{prop_B+Cindensity} involves local restrictions with respect to a fixed \Folner{} sequence, the following two propositions consider more global assumptions on $A$ and $W$.

\begin{definition}
\label{def:thick_synd_ps}
Recall that a set $T\subset\N$ is \define{thick} if it contains arbitrarily long blocks of consecutive integers.
The properties of being thick and being syndetic are dual in the following sense: a set is thick if and only if it has nonempty intersection with every syndetic set, and vice versa.
A set $P \subset \N$ is \define{piecewise syndetic} if there is a syndetic set $S \subset \N$ and a thick set $T \subset \N$ such that $P = S \cap T$.
\end{definition}

\begin{proposition}
\label{prop:thick_C_repository}
Let $W\subset\N$. 
The following are equivalent.
\begin{enumerate}
\item
\label{item1ofthepropositiononequivalentcontainment3}
For every thick set $A\subset\N$ there exist infinite sets $B\subset\N$, $C\subset W$ such that $B+C\subset A$.
\item
\label{item1.5ofthepropositiononequivalentcontainment3}
For every piecewise syndetic set $A\subset\N$ there exist infinite sets $B\subset\N$, $C\subset W$ such that $B+C\subset A$.
\item
\label{item2ofthepropositiononequivalentcontainment3}
$W$ is syndetic.
\end{enumerate}
\end{proposition}
\begin{proof}
The implication
$
\eqref{item2ofthepropositiononequivalentcontainment3}
\Rightarrow
\eqref{item1.5ofthepropositiononequivalentcontainment3}
$
follows from \cref{prop_syndeticB+C} and the implication $
\eqref{item1.5ofthepropositiononequivalentcontainment3}
\Rightarrow
\eqref{item1ofthepropositiononequivalentcontainment3}
$ is trivial.
For the final direction $
\eqref{item1ofthepropositiononequivalentcontainment3}
\Rightarrow
\eqref{item2ofthepropositiononequivalentcontainment3}
$, assume that $W$ is not syndetic and so for every $n\in\N$, there exist $a_n,b_n\in\N$ with $|b_n-a_n|\geq 2n$ and such that $W\cap [a_n,b_n]=\emptyset$. Then the set  $A=\bigcup_{n\in\N} [a_{n}+n,b_{n}]$ is a thick set with the property that any shift of it $A-t$ has finite intersection with $W$. Therefore one can not find  infinite sets $B\subset\N$, $C\subset W$ such that $B+C\subset A$.
\end{proof}

\begin{proposition}
\label{prop:syndetic-sum}
Let $W\subset\N$. 
The following are equivalent.
\begin{enumerate}
\item
\label{item1ofthepropositiononequivalentcontainment2}
For every syndetic set $A\subset\N$ there exist infinite sets $B\subset\N$, $C\subset W$ such that $B+C\subset A$.
\item
\label{item2ofthepropositiononequivalentcontainment2}
$W$ is infinite.
\end{enumerate}
\end{proposition}
\begin{proof}
The implication
$
\eqref{item1ofthepropositiononequivalentcontainment2}
\Rightarrow
\eqref{item2ofthepropositiononequivalentcontainment2}
$
is obvious.
The converse implication uses Ramsey's theorem.
Assuming $W$ is infinite and $A$ is syndetic, find $r\in\N$ such that $A\cup(A-1)\cup\cdots\cup(A-r)=\N$.
Then color each pair $\{w_1,w_2\}$ of elements of $W$ with the color $i\in\{0,\dots,r\}$ such that $w_1+w_2+i\in A$.
Using Ramsey's theorem one can extract an infinite subset $W'\subset W$ such that every pair $\{w_1,w_2\}\subset W'$ has the same color $i$.
Next split $W'$ into two disjoint infinite sets $W'=C\cup B'$ and let $B=B'+i$.
It follows that $C\subset W$ and $B+C\subset A$.
\end{proof}

\subsection{Combinatorial obstructions to $B+B+t$}
Recall \cref{thm_BBt} which guarantees that every set of positive upper Banach density contains a translate of a 
sumset $\{b_1+b_2:b_1,b_2\in B,b_1\neq b_2\}$ for some infinite set $B\subset\N$.
\cref{eg:not2b} shows that the restriction $b_1 \ne b_2$ is necessary.
More generally, examples in which one cannot find a configuration of the form $\{b ,2b\}$, or perturbations thereof such as  $\{b + b', 2b \}$, provide obstructions to many reasonable questions one might pose about sumsets in positive density sets. 
Nonetheless, beyond a particular density threshold, no such obstructions arise and the following example closes in on these limitations.

\begin{example}
\label{example_nextexamplebutnowitsgone}
For $x\geq 0$, let $\{x\}$ denote the fractional part of $x$.  
Consider the set 
\[
A:=\Big\{m\in\N:\big\{\theta\log_2(m)\big\}\in U\Big\}
\]
for any $\theta>0$ and any set $U\subset\T$ for which $U$ and $U+\theta$ (calculated $\bmod\,1$) are separated by a positive distance.
Taking $\theta=1/2$ and $U=\big[0,\tfrac12\log_2(c)\big)$ for some $1 < c< 2$ we recover the set in~\eqref{eq_archemedianobstructions} of \cref{eg:not2b}.

We claim that $A$ does not contain $B+B+t$ for any infinite set $B\subset\N$ and shift $t\geq 0$.
Indeed, for any infinite set $B\subset\N$ and $t\geq 0$, we can take $b'\in B$ arbitrary and then choose $b\in B$ sufficiently large so that the difference
\[
\left|\log_2\left(\frac{2b+t}{b+b'+t}\right)-1\right|
\]
is less than the distance $d(U,U+\theta)/\theta$.
Multiplying by $\theta$ and reducing modulo one, it follows that one can not have both $b+b'+t$ and $2b+t$ in $A$.

Taking $U=[0,1/2-\varepsilon]$ for small $\varepsilon>0$ and $\theta=1/2$, we have that 
\[
\udens(A)=\frac23\cdot(2-4^\varepsilon)
\]
As $\varepsilon\to0$ we obtain an example $A$ with upper density arbitrarily close to $2/3$.
On the other hand, taking  $U=[0,1/2-\varepsilon]$ and setting $\theta=1/2+n$ for some large $n\in\N$ (this yields the same set $A$ as setting $\theta=1/2$ and $U=T^{-1}[0,1/2-\varepsilon]$ where $T\colon\T\to\T$ is multiplication by $2n+1$) one can check that 
\[
\ldens(A)=\frac{2^{(1/2-\varepsilon)/\theta}-1}{2^{1/\theta}-1}.
\]
As $\varepsilon\to0$ and $\theta\to\infty$ it follows that $\ldens(A)$ comes  arbitrarily close to $1/2$.
\end{example}

This leads us to formulate a conjecture on the threshold for the density of obstructions.
\begin{conjecture}
\label{question_unrestricted_density_greater_than_half}Let $A\subset\N$.
\begin{enumerate}
    \item If $\udens(A)>2/3$,  there is an infinite set $B\subset\N$ such that $A\supset B+B$.
    \item If $\ldens(A)>1/2$,  there is an infinite set $B\subset\N$ such that $A\supset B+B$.
\end{enumerate}
\end{conjecture}
After we circulated our preprint, Koussek and  Radic understood the situation completely in~\cite{KR}: the conjecture is false as written but holds if one allows a shift of the set $A$  and, moreover, these results are sharp.  Additionally, they show that if one raises the upper density to $5/6$ or the lower density to $3/4$, the conclusion holds without the shift. 
We are left with a related question as to what occurs if we only assume, for example, that $|A\cap[N]|>3/4N+\log N$ for infinitely many $N$.

More generally, it would be interesting to know what properties of $A \subset \N$ guarantee that it contains $B+B+B$ for some infinite set $B \subset \N$. In particular, does any density assumption on $A$ suffice?
See also \cref{sec:general-equations} below for related questions.

We finish with a brief mention of partition regularity of $B+B$ with $B \subset \N$ infinite.
Hindman~\cite[Section 2]{Hindman1979-another} produced a three-coloring of the integers that does not admit a monochromatic set of the form $B+B$ for any infinite set $B\subset\N$, but the following question of J.\ C.\ Owings is still open.

\begin{question}[{\cite[Problem~E2494]{AmerMathMon1974Oct}}]
Is it true that, no matter how $\N$ is partitioned into two sets, one of the sets must contain $B+B$ for some infinite set $B \subset \N$?
\end{question}
It follows from the solution of 
\cref{question_unrestricted_density_greater_than_half}  by Koussek and Radic~\cite{KR} that both colors must have density $1/2$, and their paper has some further partial results in this direction.

\subsection{Ordered sums}
\label{sec:ordered}

In this section we consider generalizations of $B \oplus B$ such as
\begin{equation}
\label{eqn_poly_ordered_sumset}
\big\{f(b_1)+g(b_2): b_1,b_2\in B,~b_1<b_2\big\},
\end{equation}
where $f$ and $g$ are polynomials over $\N$.
Note that when $f(x)=g(x)=x$, the expression in~\eqref{eqn_poly_ordered_sumset} is precisely the sumset $B\oplus B$.
Although we do not formalize the terminology, we think of expressions like~\eqref{eqn_poly_ordered_sumset} as an ``ordered sum'' of $f(B)$ with $g(B)$.
It is immaterial whether one interprets the ordered sum of $f(B)$ with $g(B)$ as
\[
\big\{ f(b_1) + g(b_2) :b_1,b_2\in B,\, b_1 < b_2 \big\} \textup{ or }
\big\{ r + s : r\in f(B),~s\in g(B),~ r < s \big\}
\]
because one can obtain either as a subset of the other by refining $B$.
One could also consider ordered sumsets of unrelated sets $B$ and $C$; some results and open questions on such sumsets are collected in \cref{sec:szemeredi-and-sumset}.

In the following example we show that the restriction on the ordering $b_1< b_2$ cannot in general be replaced with $b_1\neq b_2$, even when $f$ and $g$ are both linear.

\begin{example}\label{proposition_thesameexampleagain}
Fix real numbers $\lambda_1 > \lambda_2 > 0$.
We build, using the set in \cref{example_nextexamplebutnowitsgone} as a model, a finite coloring of $\N$ such that for any infinite set $B\subset\N$ and any shift $t\geq 0$, the set
    $$\big\{\lfloor\lambda_1b_1+\lambda_2b_2\rfloor:b_1,b_2\in B,b_1\neq b_2\big\}+t$$
    is not monochromatic.
We define a finite coloring $\chi$ of the integers by setting
\[
\chi(n)=\lfloor\log_{\rho}(n)\rfloor\bmod 3
\]
where $\rho=\sqrt{\lambda_1/\lambda_2}>1$.
Fix $t\in\N$ and $B\subset\N$ infinite and let $b_1\in B$ be arbitrary.
Let $b\in B$ be sufficiently large depending on $\lambda_1,\lambda_2,t$ and $b_1$.
Let $m_1=\lfloor\lambda_1b_1+\lambda_2b\rfloor+t$ and $m_2=\lfloor\lambda_1b+\lambda_2b_1\rfloor+t$.
Note that $\rho^2m_1\approx m_2$ as $b\to\infty$, so
\[
\rho m_1<m_2 <\rho^3m_1
\]
if we choose $b$ large enough.
Taking $\log_{\rho}$ on each term we conclude that
\[
\log_{\rho}(m_2)\in\Big(\log_{\rho}(m_1)+1,\log_{\rho}(m_1)+3\Big)
\]
and after taking floors and reducing modulo  3 we get $\chi(m_1)\neq\chi(m_2)$.
\end{example}

Note that Ramsey's theorem implies~\eqref{eqn_poly_ordered_sumset} is partition regular.
This motivates the question whether instances of~\eqref{eqn_poly_ordered_sumset} can be found (up to a shift) in every set with positive density. 
Our first conjecture deals with the special case when $f$ and $g$ are linear polynomials.

\begin{conjecture}
\label{conj:k_box_j}
If $A\subset\N$ has positive upper Banach density and $\ell,m \in\N$ then there exist an infinite set $B = \{b_1 < b_2 < \cdots \}$ and a shift $t\geq 0$ such that
\begin{equation}
\label{eqn:ordered_sum}
\{ \ell b_i + m b_j : i < j \}
\end{equation}
is contained in $A-t$.
\end{conjecture}

If $\ell = m $ then \cref{conj:k_box_j} follows from \cref{thm_BBt}. 

The following question asks whether it is possible to extend \cref{conj:k_box_j} in a way that would also contain \Szemeredi{}'s theorem as a special case (see \cref{sec:szemeredi-and-sumset} for another way to combine Szemer\'edis's theorem with sumsets).

\begin{question}
\label{q_sz_waffle_sum}
Fix $k \in \N$.
If $A\subset\N$ has positive upper Banach density, then does there exists an infinite set $B=\{b_1< b_2<\cdots\}$ and a shift $t\geq 0$ such that the set
\begin{equation}
\label{eqn_waffle_sz1}
\bigcup_{\ell=0}^k \{ \ell b_i+ b_j: i < j \} 
\end{equation}
is contained in $A-t$?
\end{question}

If in \cref{q_sz_waffle_sum} one asks for the set $B$ to be arbitrarily large but finite, then a positive answer follows immediately from \Szemeredi{}'s theorem, by taking $b_i=ib$ and $t=a$ for some long arithmetic progression $a,a+b,\dots,a+Lb$ contained in $A$.
It also follows from Rado's theorem~\cite{Rado33} that whenever $\N$ is partitioned into finitely many pieces, at least one of the pieces contains~\eqref{eqn_waffle_sz1}
for an arbitrarily long but finite sequence $b_1<b_2<\cdots<b_N$ without the need for a shift. 
For $B$ infinite, even the partition variant of \cref{q_sz_waffle_sum} is not known to us; in fact we do not even know whether in every finite partition of $\N$ there exists a cell containing a configuration $\{b_j,2b_i+b_j:i<j\}$.

Surprisingly,  \cref{q_sz_waffle_sum} has a negative answer when~\eqref{eqn_waffle_sz1} is replaced by
\begin{equation}
\label{eqn_waffle_sz1_negative}
\bigcup_{\ell = 0}^k \{ b_i + \ell b_j : i < j \} 
\end{equation}
because otherwise we would find $\{ b' + b, b'+2b \}$ within $A-t$ where $b$ is much larger than $b'$, but these are the sort of configurations that \cref{eg:not2b} precludes.
The same reasoning shows that there are finite colorings of $\N$ that do not admit monochromatic sets of the form~\eqref{eqn_waffle_sz1_negative}.

This leads us to inquire which ordered  sumset configurations are actually possible.

\begin{question}
\label{q:affine_images_of_sumsets}
Let $k\in\N$ and $(a_0,d_0),\ldots,(a_k,d_k) \in \N^2$.
What are necessary and/or sufficient conditions on $(a_0,d_0),\ldots,(a_k,d_k)$ such that for any set $A\subset\N$ with positive upper Banach density there exist an infinite set $B =\{b_1< b_2<\cdots\}$ and a shift $t\geq 0$ with
\[
\bigcup_{\ell=0}^k \{ a_\ell b_i + d_\ell b_j : i < j \}
\]
contained in $A-t$?
\end{question}

Again, the order should play a role, as 
 we expect that~\eqref{eqn_waffle_sz1} appears, up to a shift, in every positive density set, when we know that~\eqref{eqn_waffle_sz1_negative} does not.
The reason is that when one considers ordered sumsets, the  first and second summands seem to follow different rules. 
We expect that the rules for restricting the first summand are similar to those for restricting the gaps in sets of positive density.
The source of these rules can be understood dynamically as refinements of \Poincare{} recurrence;  more on this perspective is discussed in  \cref{sec:recurrence-and-sumset} and \cref{thm_sumset_recurrence} in particular.
On the other hand, the second summand can sometimes be drawn from the given set of positive density, but  
there do not seem to be any arithmetic restrictions that we can place on it.

The Furstenberg-\Sarkozy{} theorem~\cite{Furstenberg-1977,sarkozy78} states that every set of positive density contains a square difference.
This motivates the following question.

\begin{question}
\label{q_sumset_sarkozy}
If $A\subset\N$ has positive upper Banach density, must there exists an infinite set $B=\{b_1< b_2<\cdots\}$ and $t\geq 0$ such that 
\begin{equation}
\label{eqn:waffle_fs}
\{b_i^2+b_j: i<j\}
\end{equation}
is contained in $A-t$?
\end{question}

After we circulated our preprint, Ackelsberg~\cite{ackelsberg} found a counterexample by building a set $A$ with density arbitrarily close to 1 such that none of its shifts contains a set of the form~\eqref{eqn:waffle_fs}.

Since the configuration~\eqref{eqn:waffle_fs} is a special case of~\eqref{eqn_poly_ordered_sumset}, an application of Ramsey's theorem shows that it is partition regular.
By the same argument, the pattern $\{b_i + b_j^2 : i<j\}$ is also partition regular.
However, the next example shows that in \cref{q_sumset_sarkozy} one cannot replace~\eqref{eqn:waffle_fs} with $\{b_i+b_j^2: i<j\}$, lending  further credence to the heuristic that we can not place arithmetic restrictions on the second summand.

\begin{example}
\label{ex_reverse_restricted_sumset_squares_is_false}
Consider the set
\[
A=\N\setminus \bigcup_{k\in\N} \big(k^2-\log(k), k^2+\log(k)\big)
\]
which has full density, yet no shift of it contains $\{b_i^2+c_j: i>j\}$ where $B=\{b_1< b_2< \cdots \}$ and $C=\{c_1< c_2< \cdots \}$ are infinite. Picking $b_i=c_i$ shows that the order in \cref{q_sumset_sarkozy} cannot be reversed, and picking $c_i=b_i^2$ shows that there exists a full-density set with the property that no shift of it contains the ordered sumset of an infinite set of perfect squares.

We note that the only property of the squares that was used to produce this counterexample is its sparsity. We could replace the set of squares with any set of zero density (such as the prime numbers) and obtain analogous counterexamples.
\end{example}

\cref{q_sumset_sarkozy} is inspired by the Furstenberg-\Sarkozy{} theorem, but a positive answer to the question does not imply it.
Our next question asks for a common generalization of both.

\begin{question}
\label{q_sumset_sarkozy_basepoint}
If $A\subset\N$ has positive upper Banach density, must there exists an infinite set $B=\{b_1< b_2<\cdots\}$ and $t\geq 0$ such that
\begin{equation}
\label{eqn:waffle_fs2}
\{b_j,\, b_i^2+b_j: i<j\}
\end{equation}
is contained in $A-t$?
\end{question}

We conclude this section with the following generalization of Questions~\ref{q_sumset_sarkozy} and~\ref{q_sumset_sarkozy_basepoint} inspired by the polynomial \Szemeredi{} theorem of Bergelson and Leibman~\cite{bergelson-leibman}.

\begin{question}\label{q_sumset_polynomials}
Let $p_1,\dots,p_k\in\Z[x]$ satisfy $p_\ell(0)=0$ for all $\ell=1,\dots,k$ and fix $A\subset\N$ with $\ubdens(A)>0$.
Can one find an infinite set $B = \{ b_1 < b_2 < \cdots \}$ and $t \ge 0$ such that
\begin{equation}
    \label{eqn:waffle_multiplefs2}
\big\{ p_1(b_i) + b_j,\, p_2(b_i) + b_j,\, \ldots,\, p_\ell(b_i) + b_j: i < j\big\}
\end{equation}
is contained in $A-t$?
\end{question}
\begin{remark}
It follows from Ackelsberg's work~\cite{ackelsberg} that the answers to Questions~\ref{q_sumset_sarkozy_basepoint} and~\ref{q_sumset_polynomials} are both negative. However, the corresponding partition questions remain open. More precisely, is it true that for every finite partition of $\N$, one of the cells contains a configuration of the form~\eqref{eqn:waffle_fs2} or~\eqref{eqn:waffle_multiplefs2} for some infinite set $B=\{b_1<b_2<\cdots\}$?
\end{remark}

We finish with a variant of \cref{q_sumset_sarkozy} where instead of squares one considers a product set.
\begin{question}
\label{q_sumset_itereated_sarkozy}
If $A\subset\N$ has positive upper Banach density, must there exists an infinite set $B=\{b_1< b_2<\cdots\}$ and $t\geq 0$ such that
\begin{equation*}
\{b_ib_j+b_k: i<j<k\}
\end{equation*}
is contained in $A-t$?
\end{question}
Similar to the way \cref{q_sumset_sarkozy} is generalized in  Questions~\ref{q_sumset_sarkozy_basepoint} and~\ref{q_sumset_polynomials}, one can inquire about analogous generalizations of \cref{q_sumset_itereated_sarkozy}.

\subsection{Sumsets of the form \texorpdfstring{$B+mB$}{B+mB}}
\label{sec:general-equations}

\cref{proposition_thesameexampleagain} provides for each $m\in\N$ a $3$-coloring of $\N$ without an infinite sumset of the form $B+mB$ with $B$ infinite. 
The coloring in the example depends heavily on the choice of $m$, raising the question whether this dependence is unavoidable.

\begin{question}
\label{question_withmultiplicativehints_coloring2}
Is it true that  for any finite coloring of $\N$, there is an infinite set $B \subset \N$ and some $m \in \N$ such that $B+mB$ is monochromatic?
\end{question}

It may be that the set of problematic $m$ is actually small in a multiplicative sense; we make this precise in the next question.

\begin{question}
\label{question_withmultiplicativehints_coloring}
Is it true that for any finite coloring of $\N$, there exists $q\in \N$ such that for all $m\in\N$ with $\gcd(m,q)=1$ there is an infinite set $B\subset \N$ such that  $B+mB$ is monochromatic?
\end{question}

One could also ask for a density version of \cref{question_withmultiplicativehints_coloring2}, but as the following example shows, some care is needed. 
Note that every thick (\cref{def:thick_synd_ps}) set $T \subset \N$ has $\ubdens(T) = 1$ but need not have $\udens(T) = 1$.

\begin{example}\label{example_sparser} 
There exists a thick set $A\subset\N$ with $\udens(A)>0$ such that whenever
\begin{equation}
\label{eq_runningoutoflabels}
A-t\supset\big\{b_1+mb_2:b_1,b_2\in B, b_1\neq b_2\big\}
\end{equation}
for some $m\geq 2$, $t\geq 0$, and set $B\subset\N$, then $B$ must be finite.
To construct such a set, let $(x_n)_{n\in\N}$ be a quickly increasing sequence (so that $x_{n+1}>4x_{n}^2$) and let 
\begin{equation}
    A=\bigcup_{n\in\N}[x_n,\tfrac32x_n].
\end{equation}
We observe that $\udens(A)=1/3$.
Suppose $t,m\in\N$ and $B\subset\N$ satisfy~\eqref{eq_runningoutoflabels}. 
We show that $B$ can not be infinite.
Indeed, take $b_0\in B$ and suppose that there exists $b\in B$ arbitrarily large, depending on $b_0$.
Since $b_0+mb+t\in A$ and $mb_0+b+t\in A$, there exist $n,n'\in\N$ such that $b_0+mb+t\in[x_n,\tfrac32x_n]$ and $mb_0+b+t\in[x_{n'},\tfrac32x_{n'}]$.
Since $b>b_0$ and $m\geq1$, we have $n\geq{n'}$.
If $n={n'}$, then 
$$b_0+mb+t\leq\tfrac32(mb_0+b+t)\iff b\leq\frac{3m-2}{2m-3}b_0+\frac{t}{2m-3},$$
but for large enough $b$ this is not possible.
If $n>{n'}$ then 
$$b_0+mb+t\geq x_n>4x_{n'}^2>2x_{n'}(mb_0+b+t)>(mb_0+b+t)^2>b^2,$$
whence $b(b-m)<b_0+t$, and again this is not possible for sufficiently large $b$.
\end{example}
\medskip

Just as the idea behind the construction in \cref{eg:not2b} is to avoid pairs $\{b,2b\}$, the idea behind \cref{example_sparser} is to avoid pairs $\{b,mb\}$ for all $m \in \N$ with $m\geq2$;
the set constructed in \cref{example_sparser} has the property that it only contains finitely many pairs $\{b,mb\}$ for any given $m\geq 2$.
Sets with positive upper density can exhibit even more extreme behavior. A well-known, but complicated, example of Besicovitch~\cite{Besicovitch35} produces sets with positive upper density containing no pair $\{b,mb\}$ for $b,m\in\N$ and $m\geq 2$. However, Davenport and \Erdos{}~\cite{Davenport-Erdos36, Davenport-Erdos51} showed that if one replaces upper density with the  notion of upper logarithmic density, which is more robust under dilation, then this behavior is avoided.

\begin{definition}[Logarithmic density]
\label{def_log_density}
Given a set $A\subset\N$,  we define its lower and upper logarithmic density, respectively, by the formulas
\[
\ldens_{\log}(A):=\liminf_{N\to\infty}\frac1{\log N}\sum_{n=1}^N\frac{\one_A(n)}n\quad\text{ and }\quad\udens_{\log}(A):=\limsup_{N\to\infty}\frac1{\log N}\sum_{n=1}^N\frac{\one_A(n)}n
\]
and when $\ldens_{\log}(A)=\udens_{\log}(A)$, we say that the logarithmic density of $A$ exists and denote the common value simply by $\dens_{\log}(A)$.
\end{definition}

For any set $A\subset\N$, one has $\ldens(A)\leq\ldens_{\log}(A) \leq \udens_{\log}(A) \leq \udens(A)$. 
It can be verified directly that the set produced in \cref{example_sparser} has zero upper logarithmic density, and so we ask the following. 

\begin{question}
\label{question_withmultiplicativehints}
Does every set $A\subset\N$ with $\udens_{\log}(A)>0$ contain a sumset $B+mB+t$ for some $m\geq2$, $t\geq0$ and some infinite set $B\subset\N$?
\end{question}

All our examples thus far of sets with positive density that avoid some unrestricted sumset are built at a logarithmic scale. 
We would like to know if this is the only obstruction.
One could attempt to make this precise by asking whether any set $A\subset\N$ with $\ubdens(A)>0$ and $\{\log(n):n\in A\}$ uniformly distributed modulo $1$ contains an unrestricted sumset $B+B+t$.
Unfortunately, as written it is not well posed since the sequence $\log(n)$ itself is not uniformly distributed modulo $1$.
This inconvenience can be avoided by using logarithmic averages, and we formalize this in the following question.

\begin{question}
If $A\subset\N$ has $\dens_{\log}(A)>0$ and for any interval $(a,b)\subset[0,1)$ and $\theta>0$ one has
\[
\dens_{\log}\Big(\Big\{n\in A:\big\{\theta\log n\}\in(a,b)\Big\}\Big)=\dens_{\log}(A)\cdot(b-a)
\]
then does $A$ contain a sumset $B+B+t$ for some $t\geq0$ and  infinite set $B\subset\N$?
\end{question}

\subsection{Uniformity norms and shifts}
\label{sec_uniformity_norms}

We have seen in \cref{prop:equivalences-for-shifts} that the shift $t$ in \cref{conj:bbbt} can be fully understood in terms of congruence obstructions.
One may wonder whether the shift in \cref{q:bbtbbbtbbbbt} is subject to the same restrictions.
However, the next example shows that is not the case.

\begin{example}
\label{example_shift-needed} 
For every $m\in\N$, the set $A_m:=2m\N-1$ has $\ubdens(A_m)>0$, but $A_m-t$ does \emph{not} contain $B\cup(B\oplus B)$ for any $t\in\{0,\dots,m-1\}$.
Indeed, suppose $\{b,b'\}\subset A_m-t$. Then $b,b'\equiv-1-t\bmod2m$ and so $b+b'\equiv-2-2t \bmod2m$. 
Since $-1-t\equiv-2-2t\bmod2m$ only holds when $t\equiv-1\bmod2m$, it follows that $b+b'\notin A_m-t$.
\end{example}
This example tells us that congruence obstructions are not the only obstacle.
Nevertheless, we believe that if similar, but higher order, restrictions on $A$ are imposed, then the need for a shift in \cref{q:bbtbbbtbbbbt} is avoided. 
We formalize this using uniformity norms.

Given a \Folner{} sequence $(\Phi_N)_{N\in\N}$ in $\N$, we say that a bounded function $f\colon \N\to\R$ \define{admits correlations along $\Phi$} if the limit
\[
\lim_{N\to\infty} \frac{1}{|\Phi_N|}\sum_{n\in\Phi_N} f(n+h_1)\cdots f(n+h_k)
\]
exists for all $h_1,\ldots,h_k\in\N\cup\{0\}$.
Note that for any $f\colon \N\to\R$ and any F\o lner sequence $\Phi$, there exists a subsequence $\Psi$ of $\Phi$ along which $f$ admits correlations. 
Under this assumption, the \define{local uniformity seminorms of $f$ along $\Phi$} are defined inductively by
\begin{align*}
\ghkn{f}_\un{0}{\Phi}
&
=
\lim_{N\to\infty} \frac{1}{|\Phi_N|}\sum_{n\in\Phi_N} f(n)
\\
\ghkn{f}_\un{k+1}{\Phi}^{2^{k+1}}
&
=
\lim_{H\to\infty} \frac{1}{H}\sum_{h=1}^H \ghkn{\Delta_h f}_\un{k}{\Phi}^{2^k}, 
\end{align*}
 where
\[
(\Delta_h f)(n)=f(n)f(n+h)
\]
for all $n \in \N$ and all $h \in \N \cup \{0\}$.
These types of uniformity seminorms were introduced in~\cite{HK-uniformity}, where the existence of the limits is deduced from analogous ergodic results.

A set $A\subset\N$ is \define{$\un{k}{\Phi}$-uniform with respect to the  \Folner{} sequence $\Phi$} if $\one_A$ admits correlations along $\Phi$, $\dens_\Phi(A) > 0$,  and $\ghkn{\one_A-\dens_\Phi(A)}_\un{k}{\Phi}=0$. 
(The restriction of this definition to sets with  $\dens_\Phi(A)>0$ is just to avoid trivially uniform sets such as the empty set.) 
We believe that, for uniform sets, one can take $t=0$ in \cref{q:bbtbbbtbbbbt}.

\begin{conjecture}
\label{conj_uni_shift}
Let $A\subset\N$ and $k\geq2$. 
If there exists a \Folner{} sequence $\Phi=(\Phi_N)_{N\in\N}$ 
with respect to which $A$ is $\un{k}{\Phi}$-uniform
then for all $\ell_1,\ldots,\ell_k\in\N$ there exists an infinite set $B\subset\N$ such that
\[
B^{\oplus \ell_1},~B^{\oplus \ell_2},~\ldots,~B^{\oplus \ell_k}\subset A.
\]
\end{conjecture} 
One could also formulate a stronger version of this conjecture, with the local uniformity seminorms replaced by the Gowers seminorms (see for example~\cite{Green_Tao08} for definitions).

We remark that \cref{conj_uni_shift} is false for $k=1$. This can be seen by taking $\ell_1=2$ and $A=2\N-1$, which is a $\un{1}{\Phi}$-uniform set for every \Folner{} sequence $\Phi$.
\cref{conj_uni_shift} is much simpler when $\ell_i = i$ for all $1 \le i \le k$ , and we include the proof for $k=2$. 
Without the constraint $\ell_i = i$, we do not have a proof even for $k = 2$ and we do not expect that the proof of \cref{prop_U2_implies_no_shift_needed}  can be modified to prove, for example, that when $A$ is $\un{2}{\Phi}$-uniform with respect to a \Folner{} sequence $\Phi$ it contains $B \oplus B \oplus B$ without introducing new tools. 

\begin{proposition}
\label{prop_U2_implies_no_shift_needed}
    If $A\subset\N$ is $\un{2}{\Phi}$-uniform with respect to some \Folner{} sequence $\Phi$ then there exists an infinite $B\subset A$ with $B\oplus B\subset A$.
\end{proposition}

\begin{proof}
Since $\ghkn{\one_A-\dens_
    \Phi(A)}_\un{2}{\Phi}=0$ one can show (for completeness, the proof is given in \cref{lemma_weakmixingisback} in the appendix) that the set $A$ is weakly mixing in the sense that
\[
\ubdens\Big(\Big\{n:\big|\dens_\Phi\big((A-n)\cap C\big)-\dens_\Phi(A)\dens_\Phi(C)\big|>\varepsilon\Big\}\Big)=0
\]
for any $\varepsilon>0$ and any set $C\subset\N$ for which all the implicit limits exist. 
Now choose $b_1\in A$ such that $\dens_\Phi\big((A-b_1)\cap A\big)>0$.
(Note that with respect to $\Phi$ a full-density set of $b_1\in A$ satisfies this.)
Put $C_1:=A\cap(A-b_1)$.  

Recursively, for each $n=2,3,\dots$ choose $b_n\in C_{n-1}$ such that
\[
\dens_\Phi\big(C_{n-1}\cap(A-b_n)\big)>0
\]
and set  $C_n:=C_{n-1}\cap(A-b_n)$.
It is clear that $B:=\{b_n:n\in\N\}\subset A$ and that $B\oplus B\subset A$.
\end{proof}

The following example gives a set $A$ which is $k$-step uniform but not $(k+1)$-step uniform and has the property that it contains $B\cup \cdots\cup B^{\oplus k}$ but not $B\cup \cdots\cup B^{\oplus k} \cup B^{\oplus k+1}$.
In this sense, the $k$-step seminorms can be said to be ``characteristic'' for $k$-fold sumsets but not for $(k+1)$-fold sumsets.
Our example is inspired by~\cite{FLW-2006},  where sets of $k$-recurrence that are not $(k+1)$-recurrence are constructed in a similar way (see \cref{sec:recurrence-and-sumset} for definitions and related results).

\begin{example}\label{example_nilbohrset}
For $k\in\N$, $\alpha\in\R\setminus\Q$, and $\varepsilon>0$ sufficiently small, it can be shown that the set
\[
A=\Bigl\{n\in\N: \{n^k\alpha\}\in \bigl[\tfrac{1}{2},\, \tfrac{1}{2}+\varepsilon\bigr]\Bigr\}
\]
satisfies $\|\one_A-\varepsilon\|_{\un{k}{\Phi}}=0$ and $\|\one_A-\varepsilon\|_{\un{k+1}{\Phi}}>0$ for all \Folner{}-sequences $\Phi$ (this follows from results in~\cite{FLW-2006}).
There is an infinite set $B$ such that $B\cup \cdots\cup B^{\oplus k}\subset A$ but there exists no infinite set $B$ with $B\cup \cdots\cup B^{\oplus k+1}\subset A$.
We prove this claim for the case $k=2$ and leave the general case to the interested reader.
The fact that 
\[
A=\Bigl\{n\in\N: \{n^2\alpha\}\in \bigl[\tfrac{1}{2},\, \tfrac{1}{2}+\varepsilon\bigr]\Bigr\}
\]
contains $B\cup B^{\oplus 2}$ follows from \cref{prop_U2_implies_no_shift_needed}.
In fact, since at each stage of the inductive construction in the proof of \cref{prop_U2_implies_no_shift_needed} there is a positive density set from which the members of $B$ can be drawn, there are in a sense many sets $B \subset \N$ for which $B\cup B^{\oplus 2}\subset A$.

We are left with showing that $A$ does not contain $B\cup B^{\oplus 2}\cup B^{\oplus 3}$.
Let $B\subset\N$ be infinite.
In view of Ramsey's theorem, replacing $B$ by an infinite subset of itself if needed, we can assume that for every distinct $b,b'\in B$, $\{bb'\alpha\}\approx_\varepsilon \gamma$ for some fixed $\gamma\in[0,1)$, where by $x \approx_\varepsilon y$ we mean that $|x-y| < C\varepsilon$ for some absolute constant $C > 0$.
By further refining $B$ we can assume that $\{b^2\alpha\} \approx_\varepsilon \beta$ for every $b\in B$ and some $\beta\in[0,1)$.

If $B\subset A$ it follows that $\beta\approx_\varepsilon1/2$.
If $B\oplus B\subset A$, expanding the square $(b+b')^2\alpha$ it follows that $2(\beta+\gamma)\approx_\varepsilon1/2$.
If $B^{\oplus3}\subset A$, then expanding the square $(b_1+b_2+b_3)^2\alpha$ it follows that $3\beta+6\gamma\approx_\varepsilon1/2$.
Since it is impossible that $\beta$, $2\beta+2\gamma$, and $3\beta+6\gamma$ are simultaneously close to $1/2$ modulo~$1$, it follows that $A$ cannot contain $B\cup B^{\oplus 2}\cup B^{\oplus 3}$.
\end{example}

\subsection{The density Ramsey property}
\label{sec:szemeredi-and-sumset}
The sumset results and \Szemeredi{}'s theorem both show the existence of patterns in a set of positive upper Banach density.  Seeking a natural unification of these results, the next definition is used to formulate such a statement (see \cref{q_sz_waffle_sum} for a different way to combine Szemer\'edis's theorem with sumsets).

\begin{definition}
We say that a collection $\CD \subset\{F\subset\N\colon 0 < |F|< \infty\}$ of finite sets has the \define{density Ramsey property} if for any set $A\subset \N$ with $\ubdens(A)> 0$, there exists some $X\in\CD$ such that $X\subset A$.
\end{definition}

By a compactness argument, the set $\CD$ 
has the density Ramsey property  if for all $\delta > 0$, there exists 
$N(\CD, \delta)\in\N$ such that for any interval $J\subset\N$ of length at least $N(\CD, \delta)$ and any set $E\subset J$ with $|E|\geq \delta |J|$, there is $X\in\CD$ with $X\subset E$.

In this language, \Szemeredi{}'s theorem is equivalent to the assertion that for every $k \in \N$, the set of all $k$-term arithmetic progressions has the density Ramsey property.

\begin{theorem}
\label{thm_sumset_szemeredi}
If $\CD\subset\{F\subset\N\colon 0 < |F|< \infty\}$ has the density Ramsey property, then for any $A\subset \N$ with $\ubdens(A)> 0$, there exist pairwise disjoint sets 
$B_1, B_2, \ldots\in\CD$ and pairwise disjoint sets $C_1, C_2, \ldots\in\CD$ such that $B_i+C_j\subset A$ for all $i,j\in\N$.  
\end{theorem}

Taking $\CD$ to be the collection of singletons, we obtain \cref{thm:MRR} and taking $\CD$ to be $k$-term arithmetic progressions, we obtain a generalization of \Szemeredi{}'s Theorem. 
The analogous statement also holds for higher order sumsets, meaning for any $k\in\N$, there exist $k$ sequences of disjoint sets $(B_{i,j})_{j\in\N}\in\CD$ for $i=1, \dots, k$ such that the $k$-fold sumsets all lie in $A$. We could also start with two collections $\CD_1$ and $\CD_2$ with the density Ramsey property, and select the sets $B_i$ from $\CD_1$ and $C_i$ from $\CD_2$.  However, to minimize notation, we restrict ourselves to the two-fold sumset and single collection $\CD$.

\begin{proof} 
Let  $A\subset \N$ with $\ubdens(A)> 0$.  By combining the results in~\cite[Section~2]{MRR}, we can find  $\varepsilon > 0$, a F\o lner sequence $\Phi  = (\Phi_N)_{N\in\N}$, and a set $L\subset \N$ such that for every nonempty finite set $F\subset L$, the set
$$
\bigcap_{x\in F}(A-x)\cap\Big\{n\in\N\colon \dens_\Phi\bigl((A-n)\cap L\bigr)> \varepsilon\Big\}
$$
has positive density with respect to $\Phi$. 
Let $F_1\subset F_2\subset \cdots\subset L$ be an exhaustion of $L$ by finite sets. 
For $i\in\N$, define
\begin{equation}
\label{eq:def-Qi}
Q_i = \bigcap_{x\in F_i}(A-x)\cap\Big\{n\in\N\colon \dens_\Phi\bigl((A-n)\cap L\bigr) > \varepsilon\Big\}. 
\end{equation}
Pick $N(\CD, \delta)\in\N$ such that for any interval $J\subset\N$ of length at least $N(\CD, \delta)$ and any set $E\subset J$ with $|E|\geq \delta |J|$, there exists $X\in\CD$ with $X\subset E$. 
Set $\delta_i = \dens_\Phi(Q_i)/2$ and for each $i\in\N$, let $J_i\subset \N$ denote an interval whose length is at least $N(\CD, \varepsilon\delta_i/2)$ and satisfies 
$$|J_i\cap Q_i|\geq \delta_i|J_i|.
$$
Without loss of generality, we can further assume that for each $i\in\N$, we have $\max J_i < \min J_{i+1}$.  
Then by~\eqref{eq:def-Qi}, it follows that 
\begin{enumerate}
\item \label{item:one}
$F_i+(Q_i\cap J_i)\subset A$ for all $i\in\N$.
\item 
If  $n_1< n_2< \cdots$ is an enumeration of $\bigcup_{i\in\N}(J_i\cap Q_i)$, then 
$$\dens_\Phi\bigl((A-n_k)\cap L\bigr) > \varepsilon.$$
\end{enumerate}
Applying Bergelson's Intersectivity Lemma~\cite[Theorem~2.1]{bergelson}, there exists $K\subset\N$ with $\dens(K) = \varepsilon$ such that for all nonempty finite sets $H\subset K$, we have 
\begin{equation}\label{eq:berg-cor}
\ubdens\Bigl(\bigcap_{k\in H}\bigl((A-n_k)\cap L\bigr)\Bigr) > 0.
\end{equation}
Define
\[
R_i = \{n_k\colon k\in K\}\cap J_i\cap Q_i 
\]
and let $I\subset\N$ be an infinite set such that for all $i\in I$, we have 
\[
|R_i|\geq \frac{\varepsilon\delta_i}{2}|J_i|
\]
which exists because $\dens(K) = \varepsilon$ and $|J_i\cap Q_i|\geq \delta_i|J_i|$. Since the length of $J_i$ exceeds $N(\CD, \varepsilon\delta_i/2)$, we can find $X_i\in \CD$ with $X_i\subset R_i$.
Then by property~\ref{item:one} and~\eqref{eq:berg-cor}, we conclude that 
\begin{enumerate}[resume]
\item
\label{item:three} 
$F_i+X_i\subset A$ for all $i\in I$. 
\item \label{item:four} 
For each $i\in I$, we have  
$$\ubdens\Bigl(\bigcap_{x\in X_1\cup\ldots\cup X_i}\bigl((A-x)\cap L\bigr)\Bigr) > 0.
$$
\end{enumerate}
We are now ready to inductively construct the sets $B_1, B_2,\ldots$ and $C_1, C_2, \ldots$.  

First choose $B_1\in\CD$ such that $B_1\subset L$. Choose $i_1\in I$ sufficiently large such that 
$B_1\subset F_{i_1}$, and  define $C_1 = X_{i_1}$.  By Property~\eqref{item:three}, we have that 
$$B_1+C_1\subset A.$$
Assume we have defined the sets $B_1, \ldots, B_n$ and $C_1, \ldots, C_n$ such that 
$B_1, \ldots, B_n\subset L$ and $C_1 = X_{i_1}, \ldots, C_n = X_{i_n}$ for some $i_1 < i_2 < \ldots <  i_n\in I$. 
Then by Property~\eqref{item:four}, there exists $B_{n+1}\in\CD$ such that 
\[
B_{n+1}\subset\bigcap_{x\in X_{i_1}\cup\cdots\cup X_{i_n}}\bigl((A-x)\cap L\bigr).
\]
It follows that $B_{n+1}\subset L$ and for all $1\leq k \leq n$, we have 
\[
B_{n+1}+C_k\subset A. \]
Choose $i_{n+1}\in I$ that is sufficiently large such that  
\[
B_1\cup \cdots \cup B_n\cup B_{n+1}\subset F_{i_{n+1}}
\]
and set $C_{n+1} = X_{i_{n+1}}$.  
Then it follows from Property~\eqref{item:three} that 
for all $1 \leq k \leq n+1$, we have
$$B_k+C_{n+1}\subset A, 
$$
completing the proof. 
\end{proof}

Whilst \cref{thm_sumset_szemeredi} provides a common generalization of \cref{thm:MRR} and \Szemeredi{}'s theorem, it does not include \cref{thm_BBt} as a special case.
This leads to the following question.

\begin{question}
\label{thm_BBt_szemeredi}
If $\CD\subset\{F\subset\N\colon 0 < |F|< \infty\}$ has the density Ramsey property, then is it true that for any $A\subset \N$ with $\ubdens(A)> 0$, there exist $t\in \N$ and disjoint 
$B_1, B_2, \ldots\in\CD$ such that $B_i+B_j\subset A-t$ for all $i\neq j\in\N$?  
\end{question}

Similarly, one can  formulate the analog for the $k$-fold sum of the sets.

\subsection{Sumsets and recurrence}
\label{sec:recurrence-and-sumset}

The quadruple $(X,\mathcal{B},\mu,T)$ is a \define{measure preserving system} if $(X,\mathcal{B},\mu)$ is a probability space, and $T\colon X\to X$ is a measurable and measure preserving map.
As part of his proof of \Szemeredi{}'s theorem, Furstenberg~\cite{Furstenberg-1977} introduced a general method, known as the \define{Correspondence Principle}, for translating a problem about finding configurations in sets of upper Banach density into a question  of recurrence in measure preserving systems.

A set $R\subset\N$ is a \define{set of recurrence}~\cite{Furstenberg-book} if for any measure preserving system $(X,\mathcal{B},\mu,T)$ and any $E\in \mathcal{B}$ with $\mu(E)>0$, there exist infinitely many $n\in R$ with $\mu(E\cap T^{-n}E)>0$.
A set $R\subset\N$ is a \define{set of strong recurrence}~\cite{bergelson} if for any measure preserving system $(X,\mathcal{B},\mu,T)$ and any $E\in \mathcal{B}$ with $\mu(E)>0$ one has
\[
\limsup_{n\in R} \mu(E\cap T^{-n}E)>0.
\]

Note that any set of strong recurrence is a set of recurrence, but the converse does not hold~\cite{forrest-1991}. There are many known examples of sets of  recurrence and strong recurrence (see for example~\cite{Furstenberg-BAMS}).

\begin{theorem}
\label{thm_sumset_recurrence}
Suppose that the set $R\subset\N$ is a set of strong recurrence. Then for any $A\subset\N$ with $\ubdens(A)>0$, there exist  infinite sets $B\subset R$ and $C\subset A$ such that
\[
\{b+c: b\in B,~c\in C,~b<c\}\subset A.
\]
\end{theorem}

We note that for a set $R$ with Banach density zero, the conclusion can not be strengthened to finding the unrestricted sumset $\{b+c:b\in B, c\in C\}$, 
and the obstruction to such a configuration is given by \cref{ex_reverse_restricted_sumset_squares_is_false} (see also \cref{prop_B+Cindensity}). 

\begin{proof}
Fix $A\subset\N$ with $\ubdens(A)>0$.
By the Furstenberg correspondence principle (see~\cite{Furstenberg-book}), there exist a measure preserving system $(X,\mathcal{B},\mu,T)$ and a set $E\in\mathcal{B}$ such that $\mu(E)=\ubdens(A)$ and 
\begin{equation}
\label{eqn_furstenberg_correspondence_inequ}    
\ubdens((A-n_1)\cap\ldots\cap (A-n_j))\geq \mu(T^{-n_1}E\cap\ldots\cap T^{-n_j}E)
\end{equation}
for all $j\in\N$ and all $n_1,\ldots,n_j\in\N\cup\{0\}$.
Since $R$ is a set of strong recurrence, the sequence of sets $E_n=E\cap T^{-n}E$ satisfies $\limsup_{n\in R}\mu(E_n)>0$. 
Applying Bergelson's intersectivity lemma~\cite[Corollary~2.4]{bergelson} to the collection $\{E_n:n\in R\}$, there exists an infinite subset $L\subset R$ such that for any finite nonempty set $F\subset L$, the intersection $\bigcap_{n\in F} E_n$ has positive measure.  
Combining this with~\eqref{eqn_furstenberg_correspondence_inequ},  we conclude that for any finite nonempty set $F\subset L$, the intersection
\[
A\cap\biggl(\bigcap_{n\in F} (A-n)\biggr)
\]
has positive upper Banach density. 
In particular, this intersection is infinite and we use it to inductively construct the sets $B$ and $C$.  Namely, let $\ell_1<\ell_2<\ldots$ be an enumeration of $L$ and choose $b_1:=\ell_1$ and then choose $c_1$ to be any element in the intersection $A\cap (A-b_1)$ with $c_1>b_1$. 
Let $b_2\in L$ be any element larger than $c_1$ and then choose $c_2$ to be any element in $A\cap (A-b_1)\cap (A-b_2)$ with $c_2>b_2$. 
Choose $b_3\in L$ larger than $c_2$ and take $c_3$ to be any element in the intersection $A\cap (A-b_1)\cap (A-b_2) \cap (A-b_3)$ with $c_3>b_3$. 
Continuing this procedure, we obtain sets $C\subset A$ and $B\subset R$ with $\{b+c: b\in B,~c\in C,~b<c\}\subset A$.
\end{proof}

Since the set of perfect squares is known to be a set of strong recurrence~\cite{Furstenberg-1977}, \cref{thm_sumset_recurrence} implies the following corollary, lending evidence towards a positive answer to \cref{q_sumset_sarkozy}.

\begin{corollary}
For any $A\subset\N$ with $\ubdens(A)>0$, there exist infinite sets $B\subset \N$ and $C\subset A$ such that
\[
\{b^2+c: b\in B,~c\in C,~b<c\}\subset A.
\]
\end{corollary}

If the set $R\subset\N$ in \cref{thm_sumset_recurrence} 
is only assumed to be infinite instead of a set of strong recurrence, then a similar argument (using Furstenberg Correspondence and Fatou's lemma) shows that for any $A\subset\N$ with $\ubdens(A)>0$, there exist  infinite sets $B\subset R$ and $C\subset\N$ such that
\[
\{b+c: b\in B,~c\in C,~b<c\}\subset A.
\]
The difference in this statement and that of \cref{thm_sumset_recurrence} is that $C$ cannot necessarily be taken from the set $A$.  On the other hand, it is not hard to see that the notion of recurrence is a necessary condition for the stronger conclusion in the theorem.  
 This motivates us to ask how much we can weaken the dynamical assumption on the set $R$ and obtain the same conclusion.  

\begin{question}
\label{q:weak}
Does the conclusion of \cref{thm_sumset_recurrence} hold if we replace the assumption on $R$ being a set of strong recurrence by $R$ being a set of recurrence? 
\end{question}

We can also obtain a higher order extension of \cref{thm_sumset_recurrence}, involving the analog for $k$-fold strong recurrence.  
Given $k\in\N$, a set $R\subset\N$ is a \define{set of strong $k$-recurrence} if for any measure preserving system $(X,\mathcal{B},\mu,T)$ and any $E\in \mathcal{B}$ with $\mu(E)>0$ one has
\[
\limsup_{n\in R} \mu(E\cap T^{-n}E\cap \ldots\cap T^{-kn})>0.
\]

With slight modifications of the proof of \cref{thm_sumset_recurrence}, one can obtain the following generalization of \cref{thm_sumset_recurrence} involving sets of strong $k$-recurrence. This gives a weaker version of \cref{q_sz_waffle_sum}, yielding a different way of combining sumsets and  \Szemeredi{}'s Theorem (for which we can ask the higher order analog of \cref{q:weak}).

\begin{theorem}
\label{thm_restricted_different_sumset_szemeredi}
Suppose $R$ is a set of strong $k$-recurrence. Then for any $A\subset\N$ with $\ubdens(A)>0$ there exist infinite sets $B\subset R$ and  $C\subset \N$ such that
\[
\bigcup_{\ell=0}^k\{\ell b+c: b\in B,~c\in C,~b<c\}\subset A.
\]
\end{theorem}

\section{Sumsets in sets of integers without  density}
\label{sec:more-general-infinite}

\subsection{Sumsets in the primes}
\label{subsection_primes}

The set $\P$ of primes has zero density, and so neither \Szemeredi{}'s theorem nor any of the results in \cref{sec:sumsets_integers} are applicable.  
However, in a major breakthrough, Green and Tao~\cite{Green_Tao08} proved the analog of \Szemeredi{}'s theorem in the primes, showing that $\P$ contains arbitrarily long arithmetic progressions. 
Their proof 
utilizes a transference method, adapting techniques developed for the  study of configurations in sets of positive density to sets of integers that are sparser but sufficiently pseudo-random.
This motivates considering sumset phenomenon in the primes and, more generally, studying which infinite patterns occur in $\P$.

\begin{conjecture}
\label{q:bcprimes}
The set of primes contains a sumset $B+C$ for some infinite sets $B,C\subset\N$.
\end{conjecture}

Granville~\cite{Granville90} proved \cref{q:bcprimes} conditionally on the prime tuples conjecture of Dickson, Hardy and Littlewood~\cite{Dickson1904, Hardy-Littlewood23} (see \cref{conj_hardylittlewood} in Appendix~\ref{appendixDHL} for the statement).
The fact that $\P$ contains a sumset $B+C$ where $B$ is infinite and $C$ has two elements is equivalent to Zhang's theorem~\cite{Zhang14} on bounded gaps in primes. The subsequent work of Maynard~\cite{Maynard15} and Polymath 8~\cite{Polymath14} implies that for any size $k$, the primes contain a sumset $B+C$ for some infinite set $B\subset\N$ and some set $C\subset\N$ with at least $k$ elements.
Recently, Tao and Ziegler~\cite{TZ23} adapted Maynard's sieve to show that there are infinite sets $B,C \subset \N$ such that the ordered sumset 
\[
\{ b + c : b \in B,~c \in C,~b < c \}
\]
is contained in the primes.
They further show  that, conditionally on  the Dickson-Hardy-Littlewood conjecture, $\P-1$ contains a sumset of the form $\{b_1+b_2: b_1,b_2\in B,~b_1\neq b_2\}$ for some infinite set $B\subset\N$.

With the goal of understanding which infinite patterns are contained in the primes, and heeding the examples discussed in  \cref{sec:sumsets_integers}, we first observe that no shift of $\P$ contains an IP-set.
Recall the notation $\FS$ defined in~\eqref{def:FS}.

\begin{proposition}
\label{prop_noprimeshiftedIP}
There is no sequence $(x_n)_{n=1}^\infty$ and shift $t\in\Z$ such that
\begin{equation}\label{eq_primestraus}
    \P+t\supseteq \FS\big(\{x_n:n\in\N\}\big).
\end{equation}
\end{proposition}
\begin{proof}
Suppose, for the sake of a contradiction, that $t\in\Z$ and $(x_n)_{n=1}^\infty$ satisfy~\eqref{eq_primestraus}.
Then $p_n:=x_n-t$ is prime for every $n\in\N$. 
Find a set of indices $F\subset\N$ with size $|F|=p_1$ and $\min F>1$, such that every prime $p_i$ with $i\in F$ has the same 
congruence class $\bmod\,p_1$.
Then on the one hand 
$$q:=(p_1+t)+\sum_{i\in F}(p_i+t)\in\P+t$$
by assumption. 
On the other hand, $q\equiv t\bmod p_1$, so $q-t$ is both a prime and a multiple of $p_1$, a contradiction.
\end{proof}
For comparison, it follows from~\cite[Example 9]{Green_Tao10} that for any $k\in\N$ the set $\P-1$ contains $\FS(\{x_1,\ldots,x_k\})$ for some set $x_1,\ldots,x_k\in\N$ (and the same holds for $\P+1$).

Inspired by \cref{q:bbtbbbtbbbbt}, we make the following conjecture, and would be interested to know which other configurations (for example, those in \cref{q_sumset_sarkozy}) can be found in $\P - 1$.

\begin{conjecture}\label{conjecture_primestrueHindman}
    For every $k\in\N$, there exists an infinite set $B\subset\N$ such that 
    $$\P-1\supset \left\{ \sum_{n \in F}  n :   F \subset B\textup{ with } 0< |F| \leq k \right\}.$$
The analogous statement holds with $\P-1$ replaced by $\P+1$.
\end{conjecture}

Observe that the shift in \cref{conjecture_primestrueHindman} is fixed independent of $k$, and any shift other than $\pm1$ leads to a false statement.
This differs from the case of sets with positive density, where the Straus example shows that the shift must be allowed to depend on $k$.
\cref{conjecture_primestrueHindman} can be proved conditionally on the Dickson-Hardy-Littlewood conjecture following similar arguments to those of Granville~\cite{Granville90} and of Tao and Ziegler~\cite{TZ23}. We present this argument in Appendix~\ref{appendixDHL}.

In view of the recent result in~\cite{TZ23} on ordered sumsets inside the primes, it makes sense to wonder whether the analogous ordered version of \cref{conjecture_primestrueHindman} can be proved unconditionally. 
To be concrete, we explicitly formulate the first two cases which are  open.

\begin{conjecture}\label{conjecture_orderedprimesBCD}
    There exist infinite sets $B,C,D\subset\N$ such that 
    $$\P\supset \big\{b+c+d:b\in B,c\in C, d\in D;\, b<c<d \big\}.$$
\end{conjecture}

\begin{conjecture}\label{conjecture_orderedprimesCBC}
    There exist infinite sets $B\subset\N$ and $C\subset\P-1$ such that 
    $$\P\supset \big\{b+c:b\in B,c\in C;\, b<c\big\}.$$
\end{conjecture}
We remark that the analogous statement to \cref{conjecture_orderedprimesCBC} where $B$ is required to be contained in $\P-1$, instead of $C$, follows from the results in~\cite{TZ23} (together with the fact that there exists an infinite admissible set contained in $\P-1$).

Whenever $A\subset\N$ contains $B+C$ with $B,C \subset \N$ infinite, the set $A$ must have infinitely many pairs of terms at the same distance.
Just as in~\cite[Remark~1.4]{TZ23} this precludes a version of \cref{q:bcprimes} that applies to subsets of $\P$ with positive relative density: one can remove a subset of zero relative density from $\P$ such that the resulting set no longer has bounded gaps, and hence no infinite sumset.

We finish this section with a question in a different direction. 
By \cref{prop_B+Cindensity}, we know that not every set of positive density contains a sumset of the form $B\oplus B+t$ where $B$ is an infinite set of primes and $t\in\N$.
However, for certain sets of number theoretic origin, such as level sets of the classical Liouville function, this conclusion may hold. 
\begin{question}
Let $A\subset\N$ be the set of natural numbers that have an odd number of prime factors counting multiplicity.
Is there an infinite set $B\subset\P$ such that $B\oplus B\subset A$?
\end{question}
This is only the easiest case to state of numerous variants of this question.  
Instead we could consider integers whose number of prime factors are $a\bmod b$ for some $a\geq 0$ and $b\geq 1$, or count the prime factors in any of these classes without multiplicity, or more generally consider the level set of any aperiodic (meaning zero average in every infinite arithmetic progression) multiplicative function.

\subsection{Zero density sets}
\label{sec:general-zero-density}

Here we move away from concrete arithmetic sets like $\P$ and consider sets of zero density more generally. 
We begin by asking if infinite sumsets occur in sets of sublinear growth.

The analogous question in the case of finite configurations plays an important role in contemporary arithmetic combinatorics and the growth rate on the size of a set $A$ that guarantees the existence of various finite patterns has been extensively studied.
For example, for $3$-term progressions, there have been recent spectacular breakthroughs on the quantitative bound needed to guarantee the existence of a three term arithmetic progression~\cite{Bloom-Sisak, Kelley-Meka}.
For progressions of higher length, the best known bounds are due to Gowers~\cite{Gowers01}, but \Erdos{} conjectured that if the sum of the reciprocals of a set of integers diverges then the set contains arbitrarily long arithmetic progressions.  
This leads us to consider similar questions for infinite sumsets instead of finite progressions.

A naive attempt is to seek a sublinear growth rate that guarantees the existence of infinite sumsets. 
In other words, can one find a non-decreasing function $h\colon\N\to\N$ with $h(N)\to\infty$ as $N\to\infty$ such that any set $A\subset \N$ with
\begin{equation}
    \label{eq:growth-rate}
|A\cap \{1,\ldots,N\}|\geq \frac{N}{h(N)}
\end{equation}
contains $B+C$ for two infinite sets $B,C\subset\N$. 
However, the answer to this question is negative, because for any such function $h$ there exists a set $A\subset \N$ whose asymptotic growth rate is larger than $N/h(N)$ but the distance between consecutive elements in $A$ tends to infinity. 
Considering two distinct elements of $C$, it follows that the sumset $B+C$ of two infinite sets $B,C\subset\N$ contains infinitely many pairs of numbers at the same distance, so such a set $A$ provides a counterexample.  
However, we ask the following.

\begin{question}\label{q_sos_4}
    Is there a set $S\subset\N$ with upper Banach density $0$ such that for any subset $A\subset S$ with positive relative density, meaning that 
    \[
    \liminf_{N\to\infty}\frac{|A\cap\{1,\ldots,N\}|}{|S\cap\{1,\ldots,N\}|}>0, 
    \]
    there exist infinite sets $B,C\subset\N$ with $B+C\subset A$?
\end{question}

Given the behavior of specific examples of sets of zero Banach density, such as the set of primes, or the set of sums of two squares, it seems that \cref{q_sos_4} has a negative answer. 
In fact it is possible that the next question has a positive answer, which would rule out a positive answer to \cref{q_sos_4}.

\begin{question}\label{q_sos_3}
    Is it true that if $S\subset\N$ has zero  Banach density, there exists a subset $A\subset S$ with positive relative density such that for all but finitely many $t\in\N$,
    $A\setminus(A-t)$ has the same relative density (in $S$) as $A$.
\end{question}

To see why a positive answer to \cref{q_sos_3} implies a negative answer to \cref{q_sos_4}, suppose that $Y_t:=A\setminus(A-t)$ has the same relative density as $A$ for all but finitely many $t\in\N$. Choosing $N_t$ sufficiently large, the set 
\[
Y:=A\setminus\bigcup_t \{a\geq N_t: a\in A-t\}
\]
also has positive relative density but no gap $t$ ``appears'' in $Y$ infinitely many times, so $Y$ cannot contain an infinite sumset.

\subsection{Quantitative versions}
While most of our focus so far has been on qualitative results, we turn now to quantitative analogs.  
We have noted that there is no growth rate on the size of the set $A$ that guarantees the existence of infinite sumsets 
(see the discussion in \cref{sec:general-zero-density}). 
In~\cite{Host}, Host exhibits an example of a set $A\subset\N$ with $\ubdens(A)>0$ such that whenever $B+C\subset A$, both $\ubdens(B)=0$ and $\ubdens(C)=0$. 
However, this does not preclude there being constraints on the growth of the summands, and so we ask whether one can impose any growth rate on the sets $B$ and $C$. 
\begin{question}
\label{q_quant_growth_B+C_in_A}
Given $\delta>0$, is there a non-decreasing function $\mathcal{H}\colon\N\to\N$ with $\mathcal{H}(N)\to\infty$ as $N\to\infty$ satisfying the following: 
for any set $A\subset \N$ with $\ldens(A)\geq \delta$ there exist $B,C\subset\N$ with
\[
\min\bigl\{|B\cap \{1,\ldots,N\}|,\,|C\cap \{1,\ldots,N\}| \bigr\}\geq \mathcal{H}(N)
\]
for all sufficiently large $N$ and such that $B+C\subset A$.
\end{question}

We do not even know if one can impose such conditions in the coloring version of this question, where one seeks a monochromatic sumset $B+C$ for an arbitrary coloring of $\N$ using finitely many (or just two) colors.\footnote{We thank Thomas Bloom for discussions surrounding this question.}
In a similar direction, a result of \Erdos{} and Galvin~\cite{erdos_galvin}
states that for any function $\mathcal{H}\colon\N\to\N$ with $\mathcal{H}(N)\to\infty$ as $N\to\infty$ there is a 2-coloring of $\N$ such that whenever the IP-set $\FS(B)$ is monochromatic, one has $|B\cap\{1,\ldots,N\}|\leq\mathcal{H}(N)$.

There is a similar regime that is worth exploring: given $\delta>0$ and $N\in\N$, let $\phi(\delta,N)$ denote the largest integer such that any set $A\subset\{1,\ldots,N\}$ with $|A|\geq \delta N$ contains $B+C$, where $B,C\subset\{1,\ldots,N\}$ satisfy $\min\{|B|,|C|\}\geq \phi(\delta,N)$.
It is not difficult to show that $\phi(\delta,N)\to\infty$ as $N\to\infty$ for any fixed $\delta>0$. On the other hand, the exact behavior of this function is not clear and we make the following conjecture. 
\begin{conjecture}\label{conjecture_phiislog}
For every $\delta>0$,
\[
\limsup_{N\to\infty}\frac{\phi(\delta,N)}{\log N}<\infty\qquad\text{ and }\qquad\liminf_{N\to\infty}\frac{\phi(\delta,N)}{\log N}>0.
\]  
\end{conjecture}
We give a heuristic explanation for why this is the correct approximation for $\phi(\delta,N)$.
Given $\delta>0$, $N\in\N$ and $A\subset[N]$ with $|A|\geq\delta N$, pick $B\subset[N]$ uniformly at random with $|B|=\varepsilon\log N$, for some small positive $\varepsilon$.
    Then the expected value of the cardinality of $C:=\bigcap_{b\in B}(A-b)$ is $|C|\approx\delta^{\varepsilon\log N}N=N^{\varepsilon\log\delta+1}\gg\log N$.
    This shows that $\phi(\delta,N)\gg\log N$.
For the converse estimate, one can take $A$ itself to be a random set, and then for any $B$ with size $|B|=\theta\log N$, for some parameter $\theta>0$, the expected value of the cardinality of $C:=\bigcap_{b\in B}(A-b)$ should be $|C|\approx\delta^{\theta\log N}N\ll\log N$.

\begin{question}
Does the limit
\[
\lim_{N\to\infty}\frac{\phi(\delta,N)}{\log N}
\]
exist, and if so, what is it?
\end{question}

We can also ask about the behavior of $\phi$ as $\delta\to0$ when  $N\to\infty$.
\begin{question}
    Is $\displaystyle\liminf_{N\to\infty}\frac{\phi\big(\delta_N,N\big)}{\log N}>0$ for some sequence $\delta_N\to0$ as $N\to\infty$?
    Can one take $\delta_N=1/\log N$?
\end{question}

\subsection{Sumsets in random sets}

Let $(p_n)_{n\in\N}$ be a sequence taking values in $[0,1]$ and consider the random set $A\subset\N$ defined such that the event $n\in A$ has probability $p_n$ and these events are independent.
A general question is for which sequences $(p_n)_{n\in\N}$ does the set $A$ contain almost surely sumsets.

If the probabilities $p_n$ are not sufficiently small, then $A$ almost surely contains arbitrarily long intervals and hence also infinite sumsets.
More precisely, if $(p_n)_{n\in\N}$ is a decreasing sequence and for every $\varepsilon>0$ we have 
\[
\lim_{n\to\infty}p_n\cdot n^{\varepsilon}\to\infty
\]
then for every $k,n\in\N$, the probability that the interval $[n+1,n+k]$ is contained in $A$ is $\prod_{i=1}^kp_{n+i}$, which is approximately equal to $p_n^k$. Since for disjoint intervals these events are independent and eventually have probability larger than $1/n$, by the second Borel-Cantelli lemma it follows that $A$ contains infinitely many intervals of length $k$ almost surely.
Since this holds for every $k$, the set $A$ is thick (see \cref{def:thick_synd_ps}) almost surely.

On the other hand, suppose that
\[
\lim_{n\to\infty}p_n\cdot n^\varepsilon =  0
\]
for some $\varepsilon>0$.
Then for some $k\in\N$, we have that  $p_n\ll n^{-1/k}$.
Then for any set $F\subset\N$ with $|F|=k$, the probability that $A$ contains infinitely many shifts of $F$ is $0$ and hence almost surely $A$ does not contain the configuration $B+F$ for any infinite set $B\subset \N$.

While random sets are often sampled from a distribution where the events $n\in A$ and $m\in A$ are independent for $n\neq m$, this is of course not a requirement. 
Random sets drawn using different distributions may still lead to interesting and non-trivial infinite sumsets. 
As an example, we ask the following question.
\begin{question}
    Fix $\gamma>0$ and, for each $x\in\R$, let $A_x:=\{n\in\N:\{n^\gamma x\}\leq1/\log n\}$.
    For which values of $\gamma$ is it true that for almost every $x\in\R$, the set $A_x$ contains a sumset $B+C$ for infinite sets $B, C\subset\N$?
\end{question}
We note that for $\gamma<1$, the set $A_x$ is thick for almost every $x$ and hence contains a sumset. 
On the other hand, when $\gamma$ is an integer one can show that for every irrational $x$ the set $A_x$ cannot contain an infinite sumset. 
The deterministic positive density analog of such sets is discussed in \cref{example_nilbohrset}.

\section{Analogs of sumsets beyond the additive integers}
\label{sec:semigroups}

\subsection{Product sets}

A sequence $(\Phi_N)_{N \in \N}$ of finite subsets of $\N$ is a \define{multiplicative \Folner{} sequence} if 
\[
\lim_{N \to \infty} \dfrac{|\Phi_N \cap t\Phi_N|}{|\Phi_N|} = 1
\]
for all $t \in \N$ where $t \Phi_N  = \{ tn  : n \in \Phi_N \}$.
It is an immediate consequence of~\cite[Theorem~1.3]{MRR} that if $A \subset \N$ has positive upper density with respect to a \Folner{} sequence on $\N$ for multiplication, then $A$ contains a product set 
\[
BC = \{ bc : b \in B, c \in C \}
\]
for some  infinite sets $B,C \subset \N$.
On the other hand, the following multiplicative analog of \cref{thm_BBt} is open.

\begin{conjecture}
If $A \subset \N$ has positive density with respect to a \Folner{} sequence for multiplication, then it contains $\{b_1b_2t:b_1\neq b_2\}$ for some infinite set $B \subset \N$ and $t\in\N$.
\end{conjecture}

The recent work of Acklesberg~\cite[Theorem 1.4]{ackelsberg} shows that this is false. He constructs an explicit example with local obstructions that is an adaptation  of the Straus example (\cref{example_straus}) to the multiplicative setting.

Next, we ask if sets with positive density but without local obstructions contain both sums and products of infinite sets. 
Say that $A\subset\N$ is \define{aperiodic} if 
\begin{equation}
\label{eqn_aperiodic}
\lim_{N\to\infty} \frac{1}{N}\sum_{n=1}^N \one_A(n)e(\alpha n) =0 \quad\textup{ for all } \alpha\in\Q\setminus\Z
\end{equation}
where $e(\alpha n) = e^{2 \pi i \alpha n}$.
If the density of $A$ exists then~\eqref{eqn_aperiodic} is equivalent to 
\[
\dens(A\cap (a\Z+b))=\frac{1}{a}\dens(A) \quad\textup{ for all } a\in\N \textup{ and all } b\in\Z,
\]
which we can think of as saying $A$ has no bias towards any infinite arithmetic progression. 
Many sets of integers with number-theoretic or combinatorial origins, such as the set of all numbers with an even number of prime factors or Beatty sequences $\{\lfloor n\alpha+\gamma\rfloor:n\in\N\}$ for some irrational $\alpha > 0$ and $\gamma \in\R$, have this property, warranting our interest in this class of sets.

\begin{question}\label{question_sumsandproducts}
Fix $A \subset \N$ that has positive density and is aperiodic. Are there infinite sets $B,C \subset \N$ with $B+C \subset A$ and $BC \subset A$?
\end{question}

Even for the explicit examples we have mentioned, we consider the answer to be interesting.
\cref{question_sumsandproducts} is related to~\cite[Question 8.4]{KMRR}, which we repeat here.

\begin{question}\label{question_sumsandproductscoloring}
Is it true that for every finite partition of $\N$, one of the sets in the partition contains $B+C$ and $BC$ for some infinite sets $B,C\subset\N$?
\end{question}

If the assumption of aperiodicity in \cref{question_sumsandproducts} is dropped, the answer is negative due to local obstructions. 
For example, the modified question is false when $A$ is the set of odd numbers.
The following variant of \cref{question_sumsandproducts} circumvents this issue by introducing a shift. 
It makes use of logarithmic density (see \cref{def_log_density}).

\begin{question}\label{question_sumsandproducts_shift}
Suppose $A\subset\N$ has positive upper logarithmic density. Are there infinite sets $B,C \subset \N$ with $B+C \subset A$ and $BC+1 \subset A$?
\end{question}
This question is a generalization of one asked previously by the second author, seeking the result when the sets $B$ and $C$ are singletons.
In \cref{question_sumsandproducts_shift} one might replace the assumption of positive upper logarithmic density with the stronger assumption of syndeticity as defined in \cref{def:thick_synd_ps}. The question is open even under that assumption, and in fact we do not know the answer to the following special case: 

\begin{question}\label{question_syndeticproducts}
    Does every syndetic set $S\subset\N$ contain a product set $BC$ for infinite sets $B,C\subset\N$?
\end{question}

When dealing with Ramsey theoretic questions involving both sums and products, it is often easier to consider the analogous questions over $\Q$ or, more generally, countably infinite fields.
To measure density in $\Q$ we make use of a \emph{double \Folner{} sequence} $\Phi=(\Phi_N)_{N\in\N}$, as introduced in~\cite[Definition 1.3]{Bergelson_Moreira17}; this means that $\Phi$ is simultaneously a \Folner{} sequence in group $(\Q,+)$ and in the group $(\Q^*,\times)$.
Such sequences lead to a notion of density, $\dens_{\Phi}$, in $\Q$ that is invariant under both addition and multiplication.

\begin{question}
Let $A\subset\Q$ have $\dens_\Phi(A)>0$ for some double \Folner{} sequence $\Phi$. 
Are there infinite sets $B,C\subset\Q$ such that $B+C \subset A$ and $BC \subset A$?
\end{question}

We end this section with the following related question of Hindman regarding partitions of $\N$.

\begin{question}[{\cite[3.1 Question~(a)]{Hindman80}}]
\label{q_hindman80_31_Qa}
Does there exist a partition of $\N$ into two sets such that neither of them contains an infinite set $B\subset\N$ together with 
\[
\{b_1+b_2: b_1,b_2\in B,~b_1\neq b_2\}\cup\{b_1b_2: b_1,b_2\in B,~b_1\neq b_2\}?
\]
\end{question}

Hindman~\cite{Hindman80} showed this question has a positive answer if one instead considers partitions of $\N$ into seven sets.

\subsection{Cartesian product sets}
We turn to Cartesian product configurations that can be found in subsets of $\N^2$ with positive density.
The starting point of our inquiry is the infinite Ramsey theorem for $2$-sets, which we recall.

\begin{theorem}[Ramsey's theorem for $2$-sets, \cite{ramsey}]
\label{thm_ramsey_2_sets}
For any finite coloring of $\N^2$ there exists an infinite set $B\subset\N$ such that
\begin{equation}
\label{eqn:bbtriangle}
\{(b_1,b_2): b_1,b_2\in B,~b_1<b_2\}
\end{equation}is monochromatic.
\end{theorem}

Motivated by Theorems \ref{thm_ramsey_2_sets} and \ref{thm_BBt}, we discuss the weakest hypothesis on a set $A \subset \N^2$ that guarantees the existence of an infinite set $B \subset \N$ and a shift $t \in \N^2$ such that 
\begin{equation}
\label{eqn:BBcart}
\big\{(b_1,b_2): b_1,b_2\in B,~b_1<b_2\big\}\subset A-t.
\end{equation}
Given $A\subset\N^2$, we write $A_n=\{m\in\N: (n,m)\in A\}$ for the fiber of $A$ with first coordinate $n$. 
A necessary condition for $A$ to satisfy~\eqref{eqn:BBcart} is that there are  infinitely many vertical fibers $A_n$ such that any finite subcollection of them has infinite intersection. 
Using this condition, the following example shows that the naive density version of \cref{thm_ramsey_2_sets} -- that a set $A \subset \N^2$ having positive density must contain a translate of~\eqref{eqn:bbtriangle} for some infinite set $B \subset \N$ -- is false.

\begin{example}
Let $\Phi=(\Phi_N)_{N\in\N}$ be a \Folner{} sequence  in $\N^2$ (see \cref{sec:amenable} for the definition).
Then there exists a set $A\subset\N^2$ with $\dens_\Phi(A)=1$ and such that each vertical fiber $A_n$ is finite.
In particular, any such $A$ cannot satisfy~\eqref{eqn:BBcart}.

For each $j\in\N$, using the F\o lner property, we choose the smallest $N_j$ such that for any $N>N_j$, the first $j$ columns of $\Phi_N$ account for less than $1/j$ of the total cardinality of $\Phi_N$. More precisely, if $N> N_j$, then 
\begin{equation*}\label{eq_folnersequencesestimate}
    \sum_{n>j}\big|(\Phi_N)_n\big|\geq|\Phi_N|(1-1/j).
\end{equation*}
Let $M_j=\max\{m:(j,m)\in\Phi_N,N\leq N_j\}=\max\bigcup_{N\leq N_j}(\Phi_N)_j$ and take $A$ to be the set described  in terms of its columns by $A_j=\{1,\dots,M_j\}$.
Note that if $n\geq j$ and $N\leq N_j$, then $A_j\supset(\Phi_N)_j$.
It follows that for $N\in[N_j,N_{j+1}]$,
$$|A\cap\Phi_N|=\sum_n|A_n\cap(\Phi_N)_n|\geq\sum_{n>j}|A_n\cap(\Phi_N)_n|=\sum_{n>j}|(\Phi_N)_n|\geq|\Phi_N|(1-1/j),$$
and hence $\dens_\Phi(A)=1$ as claimed.
\end{example}

The following question can be viewed as a density variant of \cref{thm_ramsey_2_sets}.

\begin{question}
\label{q_bxbxt2}
Suppose $A\subset\N^2$ satisfies
\begin{equation}
    \label{eq_liminfandnotsup}
\liminf_{n\in\N} \dens_\Phi(A_n) >0
\end{equation}
for some \Folner{} sequence $\Phi$ on $\N$. 
Does there exist an infinite set $B\subset\N$ and an element $t\in\N^2$ satisfying~\eqref{eqn:BBcart}?
\end{question}

Using the map $(x,y) \mapsto x+y$, one can show that a positive answer to \cref{q_bxbxt2} implies \cref{thm_BBt}.
However, a positive answer to \cref{q_bxbxt2} is significantly more powerful:  for instance, via the map $(x,y) \mapsto \ell x + my$, it  implies \cref{conj:k_box_j}.

The following example, constructed by Felipe Hern\'andez Castro,
shows the $\liminf$ in~\eqref{eq_liminfandnotsup} can not be replaced with $\limsup$.

\begin{example}
\label{ex_felipe}
Take $1<c<c'<c''<2$ and define
\[
X=\N \cap \bigcup_{n\in\N}\left[4^{n},c\cdot4^{n}\right)
\qquad\text{and}\qquad
Y=\N \cap \bigcup_{n\in\N}\left[c'4^{n},c''\cdot4^{n}\right),
\]
and take $A=X\times Y$. Note that for every integer $t\geq 0$ the intersection $X\cap (Y-t)$ is finite, implying that no shift of $A$ can contain a set of the form $\{(b_1,b_2): b_1,b_2\in B,~b_1<b_2\}$ for some infinite $B \subset \N$. Yet, there is $\delta>0$ such that $\ldens(A_n)>\delta$ holds for a set of $n\in\N$ that has positive lower density. 
\end{example}

Our next theorem provides some evidence for a positive answer to \cref{q_bxbxt2}, by showing that a weaker condition on the set $A$ implies that there exist infinite sets $B,C \subset \N$ with
\[
\{(b,c): b\in B,~c\in C,~b<c\} \subset A.
\]

\begin{theorem}
\label{q_bxc_restircted}
If $A\subset\N^2$ satisfies $\limsup_{n\in\N} \dens_\Phi(A_n) >0$ for some \Folner{} sequence $\Phi$, 
then there exist infinite sets $B,C\subset\N$ such that $\{(b,c): b\in B,~c\in C,~b<c\}\subset A$.
\end{theorem}

\begin{proof}
By assumption on the set $A$,  there are $a > 0$ and $n(1)<n(2)<\ldots\in\N$ such that $\dens_\Phi(A_{n(i)}) \ge a$ for all $i \in \N$.
It follows from~\cite[Theorem~2.1]{bergelson} that, after possibly passing to a sub-\Folner{} sequence, there is a subsequence $r(1), r(2), \ldots$ of the sequence $n(1),n(2),\ldots$ with
\[
\dens_\Phi(A_{r(1)} \cap A_{r(2)} \cap \cdots \cap A_{r(i)}) > 0
\]
for all $i \in \N$.
Inductively choose $b(i)$ and $c(i)$ with $c(i) \in A_{b(1)} \cap \cdots \cap A_{b(i)}$ and $b(i+1) > c(i) > b(i)$.
Taking $B = \{b(i) : i \in \N\}$ and $C = \{ c(i) : i \in \N \}$, the proof is complete.
\end{proof}

A tantalizing possibility is to upgrade the conclusion of \cref{q_bxc_restircted} to $A\supset B\times C$, but a counterexample is  the set $A=\{(n,m):n<m\}$ which can not contain such a Cartesian product with both $B$ and $C$  infinite.
An obvious way to rule out this example is to impose that the set $A$ must be symmetric with respect to the diagonal (meaning that $(a_1,a_2)\in A$ if and only if $(a_2,a_1)\in A$).  
However, the next example shows that even this is not sufficient.

\begin{example}\label{example_evenoddproduct}
Let $A$ be the set
$
\big\{(r,s)\in\N^2: \max\{r,s\}~\text{is even},~\min\{r,s\}~\text{is odd}\big\},
$
which is a symmetric subset of $\N^2$ whose density equals $1/4$ with respect to any \Folner{} sequence in $\N^2$. 
Moreover, $d_\Phi(A_n)=1/2$ for every odd $n$ and every F\o lner sequence $\Phi$ in $\N$.
However, there are no infinite sets $B,C \subset \N$ such that $B \times C\subset A$.
\end{example}

\subsection{Sumsets in general abelian groups}

Many questions on sums of infinite sets make sense in all countable abelian groups, where one can define a notion of upper Banach density.
With care one can even work with general abelian semigroups. Although we have done so in specific examples such as $(\N,\times)$, here we restrict our attention to countable abelian groups.

\begin{question}
Is it true in every countable abelian group $(G,+)$ that every set $A$ of positive upper Banach density contains $B_1 + \cdots + B_k$ for infinite sets $B_1,\dots,B_k \subset G$?
\end{question}

The techniques of~\cite{KMRR} together with work of Griesmer~\cite{GriesmerThesis} may suffice to give a positive answer in the case $G = \Z^d$.
In other settings one cannot proceed in the same way, because the analogous ergodic structure theory is not currently available.

\begin{conjecture}
\label{q:bbt_abelian}
Let $(G,+)$ be a countable abelian group. Every set $A$ of positive upper Banach density contains $B \oplus B + t$ for an infinite set $B \subset G$ and some $t \in G$.
\end{conjecture}
Since the circulation of a preprint of this paper, \cref{q:bbt_abelian} has been shown to be false as stated. Indeed, by combining \cite[Corollary~1.12]{CM} with \cite[Theorem~1.3]{ackelsberg} it follows that the conclusion of \cref{q:bbt_abelian} holds for an abelian group $G$ if and only if the subgroup $\{g+g:g\in G\}$ has finite index. 

When $G=\Z$, we can restrict the value of $t$ to lie in $\{0,1\}$ (see \cref{prop:equivalences-for-shifts}).
We would like to know if similar restrictions can be placed in other groups.

\begin{question}
Assuming that \cref{q:bbt_abelian} holds, what values of the shift $t$ suffice?
\end{question}

We suspect that there is a direct relation between coset representatives of the subgroup $\{g+g: g\in G\}$ and the restrictions that can be placed on the shift $t$. 
This is already hinted at by \cref{prop:equivalences-for-shifts} in $\Z$.
The case where $G$ is the direct sum of infinitely many copies of a finite cyclic group sheds some further light on this issue.
When $G$ is the direct sum of infinitely many copies of $\Z/3\Z$ one can always take $t=0$ because all elements in this group are divisible by $2$. 
In contrast, if $G$ is the direct sum of infinitely many copies of $\Z/2\Z$ then one can consider the set of all elements in $G$ whose Hamming distance to the origin is odd, yielding a set of positive density that does not contain $B\oplus B$ for infinite $B$, showing that one can not always take $t=0$ in this case.

\subsection{Analogs in general amenable groups}
\label{sec:amenable}
Many questions in combinatorial number theory have analogs in countable amenable groups.
Recall that amenability of a countable group $(G,\cdot)$ is characterized by the existence of a \define{\Folner{} sequence}: a sequence $N \mapsto \Phi_N$ of finite, nonempty subsets of $G$ satisfying
\[
\lim_{N \to \infty} \dfrac{|\Phi_N \cap g \Phi_N|}{|\Phi_N|} = 1 = \lim_{N \to \infty} \dfrac{|\Phi_N g \cap \Phi_N|}{|\Phi_N|} 
\]
for all $g \in G$.
Given a \Folner{} sequence, one speaks of the upper density of a set $A \subset G$ with respect to a \Folner{} sequence
\[
\udens_\Phi(A) = \limsup_{N \to \infty} \dfrac{|A \cap \Phi_N|}{|\Phi_N|}
\]
and can enquire about those configurations that are guaranteed in any set whose upper density is positive.
The main result of~\cite{MRR} holds in this context: every subset of a countable, amenable group with positive upper density contains the product
\[
B \cdot C = \{ bc : b \in B, c \in C \}
\]
of two infinite sets $B,C \subset G$ (\cite[Theorem 1.3]{MRR}).

At this level of generality, it is not clear how to formulate appropriate analogs of $B \oplus B + t$ and $B+C+D$.
We collect here some questions refining~\cite[Theorem 1.3]{MRR} in various directions.

\begin{question}
\label{q:false1}
If $G$ is a countable amenable group and $A\subset G$ has positive upper-density with respect to a \Folner{} sequence, is there an injective sequence $b:\N\to G$ and $t \in G$ with
\[
\{ b(i) b(j) : i < j \}
\]
contained in $t^{-1}A$?
\end{question}

\begin{question}
\label{q:false2}
If $A\subset G$ has positive upper-density with respect to a \Folner{} sequence, is there an injective sequence $b:\N\to G$ and $s,t \in G$ with
\[
\{ b(i)b(j) : i < j \} \cup \{ b(i) b(j) : i > j \} \}
\]
contained in $t^{-1} E s^{-1}$?
\end{question}

The recent results~\cite[Theorem~1.5]{CM} show that the answer to both Questions~\ref{q:false1} and~\ref{q:false2} are false,  and give partial positive results in classifying for which groups this holds. It would be interesting to find such a classification.

\begin{question}
For which countable amenable groups $G$ is it true that whenever $A\subset G$ has positive upper density with respect to a \Folner{} sequence, there are infinite sets $B,C,D \subset G$ with $BCD \subset A$?
\end{question}

By a \define{left \Folner{}} sequence we mean a sequence $N \mapsto \Phi_N$ of finite, nonempty subsets of $G$ satisfying
\[
\lim_{N \to \infty} \dfrac{|\Phi_N \cap g \Phi_N|}{|\Phi_N|} = 1
\]
for all $g \in G$.
One analogously defines \define{right \Folner{}} sequences.

\begin{question}
What are the answers to the above questions if one only assumes $A$ has positive upper density with respect to a left \Folner{} sequence or a right \Folner{} sequence?
\end{question}

In a non-abelian group $(G,\cdot)$ there are two versions of finite sumsets: given a sequence $n \mapsto g(n)$ in $G$ one can define
\begin{align*}
\fpl(g) &= \{ g(i_1) \cdots g(i_r) : i_1 < \cdots < i_r \} \\
\fpr(g) &= \{ g(i_r) \cdots g(i_1) : i_r  > \cdots > i_1 \}
\end{align*}
for the left and right finite products sets defined by $g$.
Both types of product set are partition regular by Hindman's theorem.
It is immediate that every $\fpl(g)$ contains $\{ g(i) g(j) : i < j \}$ for some unbounded sequence $n \mapsto g(n)$ in $G$.
Unlike the special case of abelian groups, in general finite products sets need not contain $BC$ with $B,C \subset G$ infinite.
In particular, we do not know how to answer the following question in full generality.

\begin{question}
Fix a countable group $(G,\cdot)$.
Does every finite partition of $G$ have a cell that contains $BC$ for infinite sets $B,C \subset G$?
\end{question}

We have drawn several parallels in the integers between results about sumsets and results about arithmetic progressions.
We conclude this section by drawing one more such parallel, posing a question on infinite patterns in cartesian products of amenable groups based on Austin's celebrated result~\cite{austin2016} on corners.
Given a \Folner{} sequence $N \mapsto \Phi_N$ on $G$ write $\Phi^d$ for the \Folner{} sequence $N \mapsto \Phi_N \times \cdots \times \Phi_N$ on $G^d$.

\begin{theorem}[{\cite[Corollary after Theorem~B]{austin2016}}]
Let $G^d$ be the direct product of $d$ copies of an amenable group $(G,\cdot)$ and let $N\mapsto \Phi_N$ be a left \Folner{} sequence in $G$.
If $A\subset G^d$ satisfies $\udens_{\Phi^d}(A)>0$ then 
\[
\begin{split}
\udens_{\Phi}\Bigl(\bigl\{& g\in G: \text{there exists } (a_1,a_2\ldots, a_d)\in G^d~\text{such that}~(a_1,a_2\ldots,a_d),
\\
&(ga_1,a_2\ldots,a_d),(ga_1,ga_2\ldots,a_d),\ldots,  (ga_1,\ldots,ga_d)\in A\bigr\}\Bigr)>0.
\end{split}
\]
\end{theorem}

This theorem gives, for example, that every set $A \subset G \times G$ with $\udens_{\Phi \times \Phi}(A) > 0$ contains many sets of the form $\{ (a,b), (ga,b), (ga,gb) \}$. The following question is an attempt to ask a sumset analog, which we only state in $G \times G$ for convenience, but which can be generalized to $G^d$ for $d> 2$ in a straightforward manner.

\begin{question}
\label{q:amenable_cartesian}
Let $(G,\cdot)$ be an amenable group and let $N\mapsto \Phi_N$ be a left \Folner{} sequence in $G$. If $A\subset G^2$ satisfies
\begin{equation}
\label{eqn:pild2}
\liminf_{g\to\infty}
\liminf_{M\to\infty}
\dfrac{1}{|\Phi_M|} \sum_{h \in \Phi_M}
\one_A(g,h)
>0,
\end{equation}
where $g \to \infty$ means leaving all finite sets, 
then there exists $(t,s)\in G^2$ and an injective sequence $b\colon \N\to G$ such that
\[
\{(t,s),~(b(i)t,s),~(b(i)t,b(j)s): i< j\}\subset A.
\]
\end{question}

If $G=\N$,  then an affirmative answer to \cref{q:amenable_cartesian} implies an affirmative answer to \cref{q_bxbxt2}.

\subsection{Ultrafilters}

We conclude with one question regarding ultrafilters on $\N$. 
An \emph{ultrafilter} is a non-empty collection $\ultra{p}$ of subsets of $\N$ closed under intersections and supersets which satisfies $A\notin\ultra{p}\iff(\N\setminus A)\in\ultra{p}$ for every $A\subset\N$. An ultrafilter is called \emph{non-principal} if all sets belonging to the ultrafilter are infinite.

Recall that one defines
\[
\ultra{p} + \ultra{q} = \{ A \subset \N : \{ n \in \N : A - n \in \ultra{q} \} \in \ultra{p} \}
\]
whenever $\ultra{p}$ and $\ultra{q}$ are ultrafilters on $\N$. 
While associative, this operation on the set of all ultrafilters is not commutative.
For exposition on ultrafilters and their role in Ramsey theory see, for instance, \cite[Section 3]{Bergelson96} or~\cite{Hindman-Strauss12}.

A set $A\subset \N$ contains $B+C$ for infinite sets $B,C \subset \N$ if and only if there are non-principal ultrafilters $\ultra{p}$ and $\ultra{q}$ with $A$ in both $\ultra{p} + \ultra{q}$ and $\ultra{q} + \ultra{p}$. 
It therefore follows from \cref{thm:MRR} that every positive density set belongs to such a pair of ultrafilters.

\begin{question}
\label{q_ultra}
For which sets $A \subset \N$ does one have
\[
A \in \ultra{p} + \ultra{q} = \ultra{q} + \ultra{p}
\]
for some pair $\ultra{p},\ultra{q}$ of non-principal ultrafilters?
\end{question}

The answer to \cref{q_ultra} is positive when $A\subset\N$ is a piecewise syndetic set (see  \cref{def:thick_synd_ps}). Indeed, in view of~\cite[Theorem 4.43]{Hindman-Strauss12} there exists $s\in\N$ such that $A-s$ contains an idempotent ultrafilter $\ultra{p}=\ultra{p}+\ultra{p}$ and hence taking $\ultra{q}=\ultra{p}+s$ yields $A \in \ultra{p} + \ultra{q} = \ultra{q} + \ultra{p}$ as desired. We do not know the answer for sets of positive density, which is perhaps the most interesting aspect of the question.

This question also makes sense in arbitrary groups, if one multiplies the ultrafilters in the appropriate way (see~\cite[Page~642]{MRR}).

\begin{appendices}

\section{Uniformity norms and weak mixing functions}

Fix a triple $(X,\mu,T)$,  where $\mu$ is a $T$-invariant probability measure on $X$. 
A function $f\in\lp^2(X,\mu)$ is \define{weak mixing} if
\[
\lim_{N \to \infty} \dfrac{1}{N} \sum_{n=1}^N | \langle T^n f, g \rangle| = 0
\]
for every $g \in \lp^\infty(X,\mu)$.
(Note that this makes sense even when the system  is not ergodic).
The function $f$ being weak mixing is equivalent to saying that
\[
\{ n \in \N : |\langle T^n f, g \rangle| < \varepsilon \}
\]
has full density for every $\varepsilon > 0$ and every $g \in \lp^\infty(X,\mu)$.
We recall here that if $f$ in $\lp^\infty(X,\mu)$ has  vanishing second Host-Kra seminorm (see~\cite[Chapter 8, Section 1]{HK-book}),
then $f$ is weak mixing.
To see this, combine
\[
\langle T^n f, g \rangle = \int \E(T^n f \cdot \overline{g} \mid \mathcal{I} ) \, d\mu,
\]
where $\E(\cdot \mid \mathcal{I})$ denotes the conditional expectation on the invariant $\sigma$-algebra $\mathcal{I}$, with the Cauchy-Schwarz inequality to obtain
\begin{align*}
\left( \lim_{N \to \infty} \dfrac{1}{N} \sum_{n=1}^N | \langle T^n f, g \rangle| \right)^2
& \le
\lim_{N \to \infty} \dfrac{1}{N} \sum_{n=1}^N \int | \E( T^n f \cdot \overline{g} \mid \mathcal{I} ) |^2 \,d \mu
\\ = 
\lim_{N \to \infty} \dfrac{1}{N} \sum_{n=1}^N
& 
\lim_{M \to \infty} \dfrac{1}{M} \sum_{m=1}^M
\int T^n f \cdot \overline{g} \cdot T^{n+m} \overline{f} \cdot T^m g \,d \mu
\end{align*}
then note that this last expression is zero whenever the second Host-Kra seminorm of $f$ is zero by~\cite[Theorem~8.13]{HK-book}.

We next present the analog of this result for the $\ell^\infty$ seminorms in \cref{sec_uniformity_norms}.
Fix a \Folner{} sequence $\Phi$, let $f\colon \Z \to \R$ be bounded and assume that $f$ admits correlations along $\Phi$ and satisfies $\ghkn{f}_\un{2}{\Phi} = 0$.
We check that
\[
\lim_{N \to \infty} \dfrac{1}{N} \sum_{n=1}^N \left| \lim_{J \to \infty} \dfrac{1}{|\Phi_J|} \sum_{j \in \Phi_J} f(n+j) g(j) \right| = 0
\]
holds for every bounded function $g \colon  \Z \to \R$.
First estimate
\begin{align*}
& \left| \lim_{J \to \infty} \dfrac{1}{|\Phi_J|}  \sum_{j \in \Phi_J} f(j+n)g(j) \right|^2
\\
\le
&{}
\limsup_{M \to \infty} \left| \dfrac{1}{M} \sum_{m=1}^M \lim_{J \to \infty} \dfrac{1}{|\Phi_J|} \sum_{j \in \Phi_J} f(j+n+m) g(j+m) f(j+n) g(j) \right|
\end{align*}
using the van der Corput  inequality.
The limit as $M \to \infty$ of the quantity inside $| \cdot |$ is non-negative, and so we may dispense in the limit with the absolute values.
Then taking the \Cesaro{} average in $N$, we obtain
\begin{align*}
&
\lim_{N \to \infty} \dfrac{1}{N} \sum_{n=1}^N \left| \lim_{J \to \infty} \dfrac{1}{|\Phi_J|} \sum_{j \in \Phi_J} f(j+n) g(j) \right|^2
\\
\le
&
\lim_{N \to \infty} \dfrac{1}{N} \sum_{n=1}^N
\lim_{M \to \infty} \dfrac{1}{M} \sum_{m=1}^M
\lim_{J \to \infty} \dfrac{1}{|\Phi_J|} \sum_{j \in \Phi_J}
f(j+n+m) g(j+m) f(j+n) g(j)
\end{align*} 
and \Holder{}'s inequality shows that this last term is zero whenever $\ghkn{f}_\un{2}{\Phi}$ equals zero.
This allows us to conclude the following.  
\begin{lemma}\label{lemma_weakmixingisback}
    If $A \subset \N$ admits correlations along $\Phi$ and $\ghkn{\one_A - \delta}_\un{2}{\Phi} = 0$, then for every $B \subset \N$ we have
\[
\dens^*(\{ n \in \N : |\dens((A - n) \cap B) - \dens(A) \, \dens(B)| > \varepsilon \})=0.
\]
\end{lemma}

\section{Conditional results in the primes}
\label{appendixDHL}

In this section we show that \cref{conjecture_primestrueHindman}, postulating infinite patterns in the primes, is implied by the well known Dickson-Hardy-Littlewood conjecture. We begin by recalling the notion of admissible set, needed to formulate the latter conjecture.

\begin{definition}\label{def_admissible}
    A (finite or infinite) set $H\subset\N$ is called \emph{admissible} if for every prime $p\in\P$ there exists $i=i(p)\in\N$ such that no element of $H$ is congruent to $i\bmod p$.
\end{definition}
\begin{conjecture}[Dickson-Hardy-Littlewood conjecture]\label{conj_hardylittlewood}
If $H\subset\N$ is finite and admissible, then there exist infinitely many $n\in\N$ such that $n+H\subset\P$.
\end{conjecture}

\begin{theorem}
\label{thm_BBtprimes1}
    \cref{conj_hardylittlewood} implies \cref{conjecture_primestrueHindman}.
\end{theorem}

We first need a lemma imposing congruence restrictions in the shift $n$ arising in \cref{conj_hardylittlewood}.

\begin{lemma}\label{lemma_DHLinprogressions}
Let $H\subset\N$ be a finite admissible set, let $q$ be a squarefree number and let $a\in\N$ such that $(a+h,q)=1$ for all $h\in H$.
Assuming \cref{conj_hardylittlewood}, there exist infinitely many $n\in\N$ such that $n+H\subset\P$ and
$n\equiv a\bmod q$.
\end{lemma}
\begin{proof}
By shifting $H$ we may assume without loss of generality that $0\in H$.
Let $N=|H|+\sum_{p|q}p$, where the sum runs over all prime divisors of $q$.
Let $\tilde Q=\prod_{p\leq N}p$ where the product is taken over all primes smaller than $N$, use the Chinese Remainder Theorem to find $L$ such that $L\equiv (\tilde Q/p)^{-1}\bmod p$ for every prime $p<N$ and let $Q=\tilde QL$.
Note that $Q/p\equiv1\bmod p$ for every prime $p<N$.

For each prime divisor $p$ of $q$, let $H_p=\big\{x\in\{0,\dots,p-1\}:x+a\not\equiv0\bmod p\big\}$, and let $$\tilde H=H\cup\bigcup_{p|q}\frac Qp\cdot H_p.$$
We claim that $\tilde H$ is admissible.
Assuming the claim for now, notice that if $n+\tilde H\subset\P$, then $n+H\subset\P$.
Moreover, if $n\not\equiv a\bmod q$, then for some prime $p|q$ we have $n\not\equiv a\bmod p$, and hence there is some $x\in H_p$ with $x+n\equiv0\bmod p$.
Therefore $h:=x\tfrac Q p\in \tilde H$ and satisfies $h+n\equiv0\bmod p$ which implies that $h+n$ is not a prime (unless $n<q$).
From \cref{conj_hardylittlewood}, we conclude that there exist infinitely many $n\in\N$ such that $n+H\subset\P$ and
$n\equiv a\bmod q$.

It remains to show that $\tilde H$ is admissible.
Note that
\[
|\tilde{H}| \le |H| + \sum_{p|q} |H_p| \le |H| + \sum_{p|q} p = N
\]
so for every prime $p>N$ we trivially have less than $p$ residue classes mod $p$ represented in $\tilde H$.
For each $p<N$, if $p$ does not divide $q$ then any $h\in\tilde H\setminus H$ is a multiple of $p$; since $0\in H$ and $H$ is admissible we conclude that there are less than $p$ residue classes mod $p$ represented in $\tilde H$.
Finally, if $p|q$, then for any $p'\neq p$, each element of $\tfrac Q{p'}\cdot H_{p'}$ is a multiple of $p$. Since no element of $H$ is congruent to $-a\bmod p$, no element of $H_p$ is congruent to $-a\bmod p$ and $\tfrac Qp\equiv1\bmod p$, we conclude that no element of $\tilde H$ is congruent to $-a\bmod p$.
\end{proof}

\begin{proof}[Proof of \cref{thm_BBtprimes1}]
Given $n,k\in\N$ denote by $[n]:=\{1,2,\dots,n\}$ and denote by $[n]^{<k}:=\{F\subset[n]:|F|<k\}$.
We stress that $[n]^{<k}$ includes the empty set.

Our goal (cf.\ \cref{conjecture_primestrueHindman}) is to construct, for a given fixed $k\in\N$, a strictly increasing sequence $b:\N\to\N$ such that
\begin{equation}\label{eq_proof_BBtprimesb}
    \forall n\in\N,\ \forall F\in[n]^{<k+1},\ F\neq\emptyset, \qquad \sum_{i\in F}b(i)\in\P-1.
\end{equation}

We construct additionally, for each prime $p$, some $x_p\in\N\cup\{0\}$ such that for every $n\in\N$ and $p\in\P$,
\begin{equation}\label{eq_proof_BBtprimes3b}
    p<k2^{n}\ \Rightarrow\ b(n)\equiv x_p\bmod p
\end{equation}
and
\begin{equation}\label{eq_proof_BBtprimes4b}
     \forall F\in[n]^{<k},\ \forall\ell\leq k-|F|,\qquad p\not|\left(1+\ell x_p+\sum_{i\in F}b(i)\right)
\end{equation}
both hold.

Let $b(1)\in\P-1$ be divisible by each prime $p< 2k$, and set $x_p:=0$ for those primes.
That such a $b(1)$ exists follows from Dirichlet's theorem.

Suppose now that $m>1$ and we have chosen $b(1),\dots,b(m-1)$ and $(x_p)_{p<k2^{m-1}}$  satisfying~\eqref{eq_proof_BBtprimesb}, \eqref{eq_proof_BBtprimes3b} and~\eqref{eq_proof_BBtprimes4b} for all $n<m$ and each prime $p<k2^{m-1}$.
For each prime $p\in(k2^{m-1},k2^{m})$ choose an arbitrary $x_p$ such that

\begin{equation}\label{eq_proof_BBtprimes5b}
    \forall F\in[m-1]^{<k},\ \forall\ell\leq k,\qquad p\not|\left(1+\ell x_p+\sum_{i\in F}b(i)\right).
\end{equation}
Note that there are at most $k2^{m-1}$ values of $x_p\bmod p$ that do \emph{not} satisfy~\eqref{eq_proof_BBtprimes5b}; therefore for each $p>k2^{m-1}$ we can choose a value of $x_p$ that does satisfy~\eqref{eq_proof_BBtprimes5b}.
Observe that this choice of $x_p$ satisfies~\eqref{eq_proof_BBtprimes4b} for all $n<m$ and $p<k2^{m}$.
Let $q=\prod_{p<k2^{m}}p$ and use the Chinese remainder theorem to find $a\in\N$ such that $a\equiv x_p\bmod p$ for every $p<k2^{m}$.
Next let
$$H:=\left\{1+\sum_{i\in F}b(i):F\in[m-1]^{<k}\right\}.$$
Note that $|H|\leq 2^{m-1}$.
It follows from~\eqref{eq_proof_BBtprimes3b} that for each $p<2^{m-1}\leq k2^{m-1}$, $|\{b(n)\bmod p:n\leq m-1\}|\leq1+\log_2(p/k)< p$ and hence that $H$ is an admissible set.
It follows from~\eqref{eq_proof_BBtprimes4b} that $(a+h,q)=1$ for every $h\in H$.

In view of \cref{lemma_DHLinprogressions}, there exists $b(m)>b(m-1)$ such that $b(m)+H\subset\P$ and $b(m)\equiv a\bmod q$.
From $b(m)+H\subset\P$ we deduce that~\eqref{eq_proof_BBtprimesb} holds with $n=m$.
From $b(m)\equiv a\bmod q$ and the fact that $a\equiv x_p\bmod p$ for every prime $p<k2^{m}$ we deduce that~\eqref{eq_proof_BBtprimes3b} holds with $n=m$.
Finally, from~\eqref{eq_proof_BBtprimes5b}, we deduce that~\eqref{eq_proof_BBtprimes4b} holds with $n=m$ for every prime $p<k2^{m}$.
\end{proof}
\end{appendices}

\bibliographystyle{amsplain}
\bibliography{expository-open-questions.bib}

\end{document}